\documentclass[12pt]{article}
\usepackage{latexsym}
\usepackage{amsmath}
\usepackage{amssymb}
\usepackage{amscd}
\usepackage{array}
\usepackage{comment}
\usepackage[all]{xy}
\usepackage{amsthm}
\usepackage{amsfonts}
\usepackage{enumerate}
\usepackage{hyperref}
\usepackage{mathtools}
\usepackage{fullpage}

\usepackage{relsize}

\usepackage{xcolor}
\usepackage{stackrel}
\newtheorem{theorem}{Theorem}[section]
\newtheorem{proposition}[theorem]{Proposition}
\newtheorem{lemma}[theorem]{Lemma}

\newtheorem*{hypothesis*}{Hypothesis}

\theoremstyle{definition}
\newtheorem{definition}[theorem]{Definition}

\newtheorem{remark}[theorem]{Remark}
\newtheorem{remarks}[theorem]{Remarks}

\newcommand\blfootnote[1]{%
	\begingroup
	\renewcommand\thefootnote{}\footnote{#1}%
	\addtocounter{footnote}{-1}%
	\endgroup
}

\newcommand{\Z}{\mathbf Z}

\newcommand{\F}{\mathbf F}
\newcommand{\C}{\mathbf C}

\newcommand{\Stab}{\mathrm{Stab}}
\newcommand{\Hol}{\mathrm{Hol}}
\newcommand{\Hom}{\mathrm{Hom}}

\newcommand{\Ker}{\operatorname{Ker}}

\newcommand{\GL}{\mathrm{GL}}

\newcommand{\Aut}{\operatorname{Aut}}

\newcommand{\Id}{\operatorname{Id}}

\newcommand{\Dic}{\operatorname{Dic}}

\providecommand{\keywords}[1]{\textbf{\textit{Keywords---}} #1}

\providecommand{\MSC}[1]{\textbf{\textit{MSC---}} #1}

\title{Determining skew left braces of size $np$}

\author{Teresa Crespo$^{(1)}$, Daniel Gil-Mu\~noz$^{(2,3,4)}$, Anna Rio$^{(5)}$ \\ and Montserrat Vela$^{(5)}$}

\date{}

\begin{document}
	
	\maketitle	
	
	\footnotesize
	
	\noindent
	(1) Departament de Matem\`atiques i Inform\`atica, Universitat de Barcelona, Gran Via de les Corts Catalanes 585, 08007, Barcelona (Spain)
	
	\noindent
	(2) Charles University, Faculty of Mathematics and Physics, Department of Algebra, Sokolovska 83, 18600 Praha 8, Czech Republic
	
	\noindent
	(3) Institut de Matemàtica, Universitat de Barcelona, Gran Via de les Corts Catalanes, 585, 08007 Barcelona, Spain

    \noindent
    (4) Dipartimento di Matematica, Università di Pisa, Largo B. Pontecorvo, 5, 56127 Pisa, Italy
    
	\noindent
	(5) Departament de Matem\`atiques, Universitat Polit\`ecnica de Catalunya, Edifici Omega, Jordi Girona, 1-3, 08034, Barcelona (Spain)
	
	\begin{abstract}
		We define the twofold semidirect product of two skew left braces, in which both the additive and multiplicative groups are semidirect products of the corresponding groups of the given skew left braces. We consider an odd prime $p$ and an integer $n$ satisfying $p\nmid n$, $p\nmid|\mathrm{Aut}(E)|$ for every group $E$ of order $n$ and such that each group of order $np$ has a unique $p$-Sylow subgroup. Under these conditions, we prove that any skew left brace of size $np$ is either a twofold semidirect product of the trivial brace of size $p$ and a skew left brace of size $n$ or a companion skew left brace of that one. We develop an algorithm to obtain all skew left braces of size $np$ from the skew left braces of size $n$ and provide a formula to count them. We use this result to describe all skew left braces of size $12p$ for $p\geq 7$, which proves a conjecture of V.G. Bardakov, M.V. Neshchadim and M.K. Yadav.
	\end{abstract}
	
	\MSC{Primary: 20B35, 20D20, 20B05. Secondary: 81R50, 16T25, 12F10}
	
	
	\blfootnote{The first author was supported
		by grant PID2019-107297GB-I00 (Ministerio de Ciencia, Innovación y Universidades).
		The second author was supported by Czech Science Foundation, grants 21-00420M and 24-11088O, and by Charles University Research Centre program UNCE/SCI/022. The second, third and fourth authors were supported by the grant PID2022-136944NB-I00 (Ministerio de Ciencia e Innovación). The first, third and fourth autors were partially supported by grant 2021-SGR-01468 (Generalitat de Catalunya).
		
		Email addresses: teresa.crespo@ub.edu, daniel.gil-munoz@mff.cuni.cz, ana.rio@upc.edu, montse.vela@upc.edu }

	\section{Introduction}
	In \cite{R07} Rump introduced an algebraic structure called (left) brace to study
	set-theoretic solutions of the Yang-Baxter equation. Guarnieri and Vendramin introduced in \cite{GV} the more general concept of skew (left) brace. A skew left brace is a 
    triple  $(B,\cdot,\circ)$, where  $(B,\cdot)$ and  $(B,\circ)$  are groups and the following compatibility condition holds:
    for all $a, b, c \in B$,
	$$a\circ(b\cdot c) = (a\circ b)\cdot a^{-1} \cdot (a\circ c),$$
	 where $a^{-1}$ denotes the inverse of $a$ in $(G,\cdot)$. We call $(B, \cdot$) the additive group and $(B, \circ)$ the multiplicative group of the skew left brace. We say that $B$ is a left brace if its additive group is abelian. The \textit{size} of a skew left brace is the cardinality of the underlying set $B$. If $(B,\cdot)$ is a group, then $(B,\cdot,\cdot)$ is a skew left brace, called the \textit{trivial skew left brace}. Note that any skew left brace of prime size is necessarily trivial.

Consider two skew braces $(B_1,\cdot,\circ)$ and $(B_2,\cdot,\circ)$. A \textit{skew left brace homomorphism} is a map $f : B_1 \to B_2$  such that $f(x\cdot y) = f(x)\cdot f(y)$ and $f(x\circ y) = f(x)\circ f(y)$ for all $x,y \in B_1$. \textit{Skew left brace isomorphisms} and \textit{automorphisms} are defined accordingly. If there is a skew left brace isomorphism between them, we say that the skew left braces $B_1$ and $B_2$ are isomorphic. We denote by $\Aut(B,\cdot,\circ)$  the group of skew left brace automorphisms of $(B,\cdot,\circ)$.

If $(B,\cdot,\circ)$ is a skew left brace, the multiplicative group acts on the additive group by
automorphisms since for each $a \in B$, the map
	$$\begin{array}{llll} \lambda_a: &B &\rightarrow &B
	\\ & b& \mapsto & a^{-1}\cdot (a\circ b) \end{array}$$
	is an automorphism of $(B,\cdot)$ and  the map  
    $\lambda: (B,\circ)\to \Aut(B,\cdot)$ given by $\lambda(a)=\lambda_a$ is a group homomorphism. The following formulas hold:  $a\circ b=a \cdot\lambda_a(b)$, $a\cdot b=a \circ\lambda_a^{-1}(b)$. An \textit{ideal} of a skew left brace $(B,\cdot,\circ)$ is $I\subseteq B$ such that $(I,\circ)$ is a normal subgroup of $(B,\circ)$, $(I,\cdot)$ is a normal subgroup of $(B,\cdot)$ and $\lambda_a(I)\subseteq I$ for all $a\in B$.
    One can quotient out skew left braces by ideals to produce new skew left braces. From the definition, it follows directly that 
    $I\cdot a= a\cdot I= a\circ I= I\circ a$
for any $a \in B$. This implies that the additive and multiplicative group structure on $B/I$ are
well-defined. The brace condition  holds in $B/I$ because it holds in $B$.

    Due to the results in \cite{CR09}, \cite{B16} and \cite{GV},
	the determination of skew left braces of a given size can be translated to group theory, since there is a bijective correspondence between skew left braces structures defined in a group  $(N,\cdot)$ and regular subgroups of $\Hol(N,\cdot)$.

    Following the establishment of this connection, a number of works have considered the problem of determining the skew braces of a specific size, for instance square-free \cite{AB21} or $p^2q$ for $p,q$ primes \cite{AB22a,AB22b,CCD20,CCD24}.
	In \cite{CGRV1}, we classified left braces of size $8p$, for $p$ a prime number, $p \neq 3,7$.
	In \cite{CGRV2}, we gave a method to determine all left braces of size $np$ from the braces of size $n$ and certain classes of group homomorphisms from the multiplicative group of these left braces of size $n$ to $\Z_p^*$. For a prime number $p\geq 7$, we applied the obtained results to classify left braces of size $12p$. In particular we proved the Conjectures 4.2 and 4.3 in \cite{BNY} on the numbers $b(8p), b(12p)$ of isomorphism classes of left braces of size $8p$ and $12p$, respectively.

    The dependence of left braces of size $np$ on left braces of size $n$ mentioned above can be precised by the notion of semidirect product of left braces. This was introduced in \cite{R08} for algebraic structures called cycloids, from which braces can be derived, and was reformulated for left braces in \cite{Cedo,CJO}. Similarly, in \cite{CCS16} the semidirect product for right braces was introduced. Bachiller \cite{B18} introduced the matched product for left braces, and was generalized to left semi-braces in \cite{CCS20}. In all these constructions (except cycloids), the direct product of the original additive groups is carried out, while some variant is introduced to define the product of the multiplicative structures. In \cite{CMS22}, a notion of product is introduced for left inverse semi-braces consisting in non-trivial products for additive and multiplicative structures, and was consequently called the double semidirect product. In group theory, the concept of a semidirect product coincides with that of a split extension. As we have mentioned, as a way to produce new constructions, in the literature different theories of extensions of braces have been considered. This has been generalized in \cite{RY24}
    to a theory of extensions of skew left braces by abelian groups. 
    
	Our aim in this paper is to generalize the method obtained in \cite{CGRV2} to skew left braces of size $np$, where $p$ is a prime number and $n$ an integer not divisible by $p$ and such that every group of order $np$ has a normal subgroup of order $p$. To do so, we use extensions of skew braces of size $n$ by the trivial brace $\Z_p$. One of them is obtained by introducing a new notion of semidirect product of two skew left braces, called the twofold semidirect product, which is a skew left brace whose additive and multiplicative groups are semidirect products of the original ones (see Definition \ref{twofoldsemidprod}). With this definition, we will prove that a skew left brace of size $np$ is either a twofold semidirect product of the trivial left brace of size $p$ and a skew left brace of size $n$ or a \emph{companion} brace of this one. Given a skew left brace $(B_n,\cdot,\circ)$ of size $n$, we will show that each pair $(\sigma,\tau)$, with $\sigma\in\Hom(B_n,\Z_p^*)$, $\tau\in\Hom((B_n,\circ),\Z_p^*)$ give rise to two skew left braces of size $np$. The number of skew left braces of size $np$ will be given in terms of the number of orbits under actions involving the morphisms $\sigma$ and $\tau$.

	We shall apply our method to classify skew left braces of size $12 p$, for $p$ a prime number, $p\geq 7$. In particular we will prove the following conjecture in \cite{BNY} on the number $s(12p)$ of isomorphism classes of skew left braces of size $12p$, $p\geq 7$.
	
	\begin{equation}\label{conjec}	
	s(12p)= \left\{ \begin{array}{ll} 324 & \text{\ if } p\equiv 11 \pmod{12}, \\
	410 & \text{\ if } p\equiv 5 \pmod{12}, \\
	606 & \text{\ if } p\equiv 7 \pmod{12}, \\
	782 & \text{\ if } p\equiv 1 \pmod{12}.
	\end{array} \right.
	\end{equation}
	
	\noindent
	We note that $s(24)=855, s(36)=400$ and $s(60)=418$ (see \cite[p. 21]{V}).
	
	\section{Skew braces of size $np$}\label{method}

    Throughout this paper, we will refer to skew left braces (resp. left braces) just as skew braces (resp. braces).

\subsection{Twofold semidirect products of skew braces}
	
	We recall the definition of direct and semidirect product of skew braces as one can see, for instance, in \cite{Cedo} and \cite{SV}. Let $A$ and $B$ be skew braces. Then $A \times B$ together with
	$$
	(a,b)\cdot(a',b')=(a\cdot a',b\cdot b'),\quad (a,b)\circ (a',b')=(a\circ a',b\circ b')
	$$
	is a skew brace called the \textit{direct product of the skew braces $A$ and $B$}. Now, let $\tau:(B,\circ)\to\Aut(A,\cdot,\circ)$
	be a group homomorphism.
	Consider in $A\times B$
	the additive structure of the direct product $(A,\cdot)\times (B,\cdot)$ and the multiplicative structure of the semidirect product $(A,\circ)\rtimes_{\tau} (B,\circ)$:
	$$
	(a,b)\cdot(a',b')=(a\cdot a',b\cdot b'),\quad 
	(a,b)\circ(a',b')=(a\circ\tau(b)(a'),b\circ b').
	$$
	Then, we obtain a skew brace,  which is called the \textit{semidirect product of the skew braces $A$ and $B$ via $\tau$}.

	We observe in particular that the additive group of the semidirect product of two braces is the direct product of the corresponding additive groups. In \cite{CGRV2} we proved that a brace of size $np$, where $p$ is a prime number and $n$ an integer not divisible by $p$ and such that every group of order $np$ has a normal subgroup of order $p$, is a direct or semidirect product of the trivial brace of size $p$ and a brace of size $n$. In this case, the additive group was a direct product of $\Z_p$ and an abelian group of order $n$. We are now interested in skew braces of size $np$, for which both additive and multiplicative groups may be strictly semidirect products of $\Z_p$ and a group of order $n$. We want then to determine the relationship between these skew braces of size $np$ and skew braces of size $n$. To this end we shall define the twofold semidirect product of two skew braces.
	
	\begin{proposition}\label{dsdp}
		Let $A$ and $B$ be skew braces and let
		$$\sigma:(B,\cdot)\to\Aut(A,\cdot,\circ), \quad \tau:(B,\circ)\to\Aut(A,\cdot,\circ)$$
		be group homomorphisms. Consider in $A\times B$
		the additive structure of the semidirect product 
        $(A,\cdot)\rtimes_{\sigma} (B,\cdot)$
        and the multiplicative structure of the semidirect product
        $(A,\circ)\rtimes_{\tau} (B,\circ)$
		$$
		(a,b)\cdot(a',b')=(a\cdot\sigma(b)(a'),b\cdot b'),\quad 
		(a,b)\circ(a',b')=(a\circ\tau(b)(a'),b\circ b').
		$$
		We assume that the following equality holds for any $a\in A$ and all $b_1, b_2 \in B$:
		\begin{equation}\label{eqdsdp}
		\sigma((b_1\circ b_2)\cdot b_1^{-1})\lambda_a=\lambda_a\tau(b_1) \sigma(b_2) \tau(b_1)^{-1}
		\end{equation} 
		Then $(A\times B, \cdot,\circ)$ is a skew brace.
	\end{proposition}
	
	\begin{proof}
		We check the skew brace condition for $(A\times B, \cdot,\circ)$, namely the equality
		$$(a_1,b_1)\circ \left( (a_2,b_2) \cdot (a_3,b_3) \right) = \left( (a_1,b_1) \circ (a_2,b_2) \right) \cdot (a_1,b_1)^{-1} \cdot \left( (a_1,b_1)\circ (a_3,b_3) \right)$$
		for $(a_1,b_1), (a_2,b_2), (a_3,b_3) \in A\times B$. The term on the left is
		$$ (a_1,b_1)\circ (a_2\cdot \sigma(b_2)(a_3),b_2\cdot b_3) \\
		= (a_1 \circ \tau(b_1)(a_2\cdot \sigma(b_2)(a_3)),b_1 \circ ( b_2\cdot b_3))$$
		and the term on the right is
		$$\begin{array}{l} 
		(a_1 \circ \tau(b_1)(a_2),b_1\circ b_2) \cdot (\sigma(b_1^{-1})(a_1^{-1}),b_1^{-1})\cdot(a_1 \circ \tau(b_1)(a_3),b_1\circ b_3)\\
		=(a_1 \circ \tau(b_1)(a_2)\cdot \sigma(b_1\circ b_2)(\sigma(b_1^{-1})(a_1^{-1})),b_1\circ b_2 \cdot b_1^{-1})\cdot(a_1 \circ \tau(b_1)(a_3),b_1\circ b_3)\\
		=(a_1 \circ \tau(b_1)(a_2)\cdot \sigma(b_1\circ b_2 \cdot b_1^{-1})(a_1^{-1})\cdot \sigma(b_1\circ b_2 \cdot b_1^{-1})(a_1 \circ \tau(b_1)(a_3)),b_1\circ b_2 \cdot b_1^{-1} \cdot b_1\circ b_3)
		\end{array}
		$$
		Since $B$ is a skew brace, the two second components are equal. Now, the first component of the term on the left is
		$$\begin{array}{l} a_1 \circ \tau(b_1)(a_2\cdot \sigma(b_2)(a_3))=a_1 \circ \left(\tau(b_1)(a_2) \cdot \tau(b_1)(\sigma(b_2)(a_3)\right)\\ 
		=a_1 \circ \tau(b_1)(a_2)\cdot a_1^{-1} \cdot a_1  \circ \tau(b_1)(\sigma(b_2)(a_3))=
		a_1 \circ \tau(b_1)(a_2)\cdot \lambda_{a_1}(\tau(b_1)(\sigma(b_2)(a_3)))
		\end{array}
		$$
		where the second equality follows from the brace condition for $A$. The first component of the right term is
		$$\begin{array}{l} a_1 \circ \tau(b_1)(a_2)\cdot \sigma(b_1\circ b_2 \cdot b_1^{-1})(a_1^{-1})\cdot \sigma(b_1\circ b_2 \cdot b_1^{-1})(a_1 \circ \tau(b_1)(a_3)) \\ = a_1 \circ \tau(b_1)(a_2)\cdot \sigma(b_1\circ b_2 \cdot b_1^{-1})(a_1^{-1}\cdot a_1 \circ \tau(b_1)(a_3))
		\\= a_1 \circ \tau(b_1)(a_2)\cdot \sigma(b_1\circ b_2 \cdot b_1^{-1})(\lambda_{a_1}( \tau(b_1)(a_3))).
		\end{array}
		$$
		The equality of the first components is then equivalent to
		$$\lambda_{a_1}\tau(b_1)\sigma(b_2)(a_3)=\sigma(b_1\circ b_2 \cdot b_1^{-1})\lambda_{a_1}\tau(b_1)(a_3)$$ for all $a_3\in A$, hence to equality \eqref{eqdsdp}.
	\end{proof}

A general theory for arbitrary extensions of skew braces should provide a common frame for the compatibility conditions among actions which allow to define a new skew brace structure.
    
	\begin{definition}\label{twofoldsemidprod}
		Let $\sigma$ and $\tau$ be group homomorphisms as in Proposition \ref{dsdp}. We say that $(\sigma,\tau)$ is a \textit{good pair} if the compatibility condition \eqref{eqdsdp} is satisfied. If $(\sigma,\tau)$ is a good pair,  we call the {\it twofold semidirect product} of the skew braces $A$ and $B$ via $\sigma$ and $\tau$ the skew brace $A\rtimes_{\sigma}^{\tau}B\coloneqq(A\times B, \cdot,\circ)$ defined in Proposition \ref{dsdp}.
	\end{definition}
	
	\begin{remarks}
		\begin{enumerate}
			\item If $\sigma$ is the trivial homomorphism, $\sigma(b)=\Id_A$ for all $b$, then 
            $(\sigma, \tau)$ is a good pair for all $\tau$   and we recover the notion of semidirect product of skew braces.
			\item If $A$ is a trivial brace, then $\lambda_a=\Id_A$ for all $a$. In that case, if  $\Aut(A)$ is abelian, then \eqref{eqdsdp} is equivalent to $\sigma(b_1\circ b_2)=\sigma(b_1)\sigma(b_2)$ for all $b_1,b_2 \in B$, namely to $\sigma$ being also group homomorphism with respect to $\circ$, that is, to $\sigma$ being a skew brace homomorphism.
		\end{enumerate}
		
	\end{remarks}
	
	\begin{lemma}\label{rm}If $(B,\cdot,\circ)$ is a skew brace, $\sigma:(B,\cdot,\circ)\to \Z_p^*$ is a skew brace homomorphism and $\tau:(B,\circ)\to \Z_p^*$ is a group homomorphism, then 
		in $\Z_p\times B$ we can define two non-isomorphic skew brace structures with the same additive group with operation
		$$
		(m,a)\cdot (n,b)=(m+\sigma(a)n, a\cdot b)
		$$
		and multiplicative group with respective operations
		\begin{equation}\label{compa}
		(m,a)\circ (n,b)=
		\left\{
		\begin{matrix}
		(m+\tau(a)n,\ a\circ b)\\
		(\sigma(b)m+\tau(a)\sigma(a)n,\ a\circ b)\\
		\end{matrix}
		\right.
		\end{equation}
	\end{lemma}
	
	The first one corresponds to the twofold semidirect product defined above 
	 and it is a straightforward computation to check the brace condition for the second multiplicative structure.  
    Alternatively, one can consider the good triplets $(\tau,\sigma,\tau^{-1})$ and $(\tau,\sigma,\sigma\tau^{-1})$, respectively, and the pair $(0,0)\in C_N^2$ to apply \cite[Theorem 3.3]{RY24}.
    In both cases, the additive group is $\Z_p\rtimes_{\sigma}(B,\cdot)$ and the multiplicative group is isomorphic to $\Z_p\rtimes_{\tau}(B,\circ)$. If $x=(m,a)$ and $y=(n,b)$, we have 
	$$
	\lambda_x(y)=x^{-1}\cdot (x\circ y)=
	(-\sigma(a^{-1})m,a^{-1})\cdot
	((m,a)\circ (n,b))=\left\{
	\begin{matrix}
	(\sigma(a^{-1})\tau(a)n,\ \lambda_a(b))\\
	(\tau(a)n-\sigma(a^{-1})m+\sigma(a^{-1})\sigma(b)m,\ \lambda_a(b))\\
	\end{matrix}
	\right.
	$$
	
\subsection{Describing skew braces of size $np$}

        From now on we fix the following assumption, which will be maintained along all the paper: $p$ is an odd  prime and $n$ an integer such that $p$ does not divide $n$ and each group of order $np$ has a normal subgroup of order $p$. In Sections \ref{twelve} and \ref{braces12p} we consider $n=12$ and $p\geq 7$, which satisfy this conditions.
	
	Our aim is to show that we can determine all skew braces of 
	size $np$ from skew braces of size $n$ as the ones described in Lemma \ref{rm}. We will do it taking into account 
	the correspondence between braces and regular subgroups of
	the holomorph. We note that, if $G$ is a regular subgroup of $\Hol(A,\cdot)=A\rtimes \Aut(A,\cdot)$, then for each $a\in G$, $(a,f)\to a$ gives a bijection between $G$ and $A$.
	
	\begin{proposition}[\cite{GV} Theorem 4.2]\label{GV}
		Let $(A,\cdot,\circ)$ be a skew brace. Then
		$\{ (a,\lambda_a) \, : \, a \in A \}$
		is a regular subgroup of $\Hol(A,\cdot)$, isomorphic to $(A,\circ)$. Conversely, if $(A,\cdot)$ is a group and $G$ is a regular subgroup of $\Hol(A,\cdot)$, then $(A,\cdot,\circ)$ is a skew brace with $(A,\circ) \simeq G$, where
		$a \circ b=a\cdot f(b),$ and $({\pi_1}_{|G})^{-1}(a)=(a,f) \in G$.
		
		These assignments give a bijective correspondence between the set of isomorphism classes of skew braces $(A,\cdot,\circ)$ and the set of conjugation classes of regular subgroups of $\Hol(A,\cdot)$.
	\end{proposition}

	For a skew brace of size $np$ we denote $N$ its additive group and $G$ its multiplicative group. Therefore, we will work with a group $(N,\cdot)$ and a regular subgroup $G$
    of $\Hol(N)$. Since \textit{Schur-Zassenhaus theorem} states that the groups of order $np$ are semidirect products, we have $N=\Z_p\rtimes_\sigma E$, with $E$
	a group of order $n$ and $\sigma:E\to \Z_p^*$ a group homomorphism. 
	Following \cite{curran}, the group automorphisms of $N$ can be described as matrices
	$$
	\begin{pmatrix}
	\alpha & \gamma\\
	0&\lambda
	\end{pmatrix}
	\mbox{ 
		with }\alpha\in \Z_p^*=\Aut(\Z_p),\ 
	\lambda\in \Aut(E)\mbox{ such that }\sigma\lambda=\sigma
	\mbox{ and }\gamma:E\to\Z_p \mbox{ a }1-\mbox{cocycle}$$
	and the action on $N$ is given by
	$$\begin{pmatrix}
	\alpha & \gamma\\
	0&\lambda
	\end{pmatrix}
	\begin{pmatrix}
	m\\
	a
	\end{pmatrix}=
	\begin{pmatrix}
	\alpha m+\gamma(a)\\
	\lambda(a)
	\end{pmatrix}.
	$$
	Under our assumption, the subgroups of order $n$ of $N$ are  conjugates of $E$ by elements in $\Z_p$ and this
    amounts to $H^1(E,\Z_p)=0$, so we are left with the coboundaries
	$$
	\gamma_i(a)=i-\sigma(a)i,\quad  i\in\Z_p.
	$$
	For a trivial $\sigma$, we get the direct product $\Aut(\Z_p)\times \Aut(E)$, and for a non trivial $\sigma$ the group $\Aut(N)$ has order $p(p-1)s$ with $s=\#\{\lambda\in \Aut(E): \sigma \lambda=\sigma\}$. This subgroup is the stabilizer $\Sigma_{\sigma}$ of $\sigma$ under the right group action
	$$
	\begin{array}{ccl}
	\Aut(E)\times \Hom(E,\Z_p^*)&\longrightarrow &\Hom(E,\Z_p^*)\\
	(g, \sigma) &\longrightarrow & \sigma g
	\end{array}.
	$$
	The orbit of $\sigma$ is the set of group homomorphisms $\sigma':E\to \Z_p^*$ such that $\Z_p\rtimes_{\sigma'}E\simeq N$.  Let us remark that for 
    $\lambda\in \Sigma_{\sigma}$ we have $\gamma_i\lambda=\gamma_i$ for all $i$.

	All together, if we write
	$$
	\Aut(\Z_p\rtimes_{\sigma}E)=\left\{ \begin{pmatrix}
	\alpha & \gamma_i\\
	0&\lambda
	\end{pmatrix}\ \colon \alpha\in\Z_p^*,\ i\in\Z_p,\ \lambda\in\Sigma_{\sigma}\right\}
	$$
	then we get a description of the holomorph
	$$
	\Hol(\Z_p\rtimes_{\sigma}E)=\left\{\Big[\begin{pmatrix}
	m\\
	a
	\end{pmatrix}, M\Big]\ \colon m\in\Z_p,\ a\in E,\ M\in\Aut(\Z_p\rtimes_{\sigma}E) \right\}
	$$
	with
	$$
	[u, M_u][v,M_v]=[u\cdot M_u v,M_uM_v ]
	$$
	where $\cdot$ indicates operation in semidirect product $\Z_p\rtimes_{\sigma}E$. The first component gives the action of $\Hol(N)$ on $N$:
    $$
    ([u, M_u],v)\longmapsto u\cdot M_uv.
    $$

 \begin{lemma}\label{ordpsubgp}
   Let $p$ and $n$ be as in the fixed assumption, let $N=\Z_p\rtimes_{\sigma}E$ be a group of order $np$, with $E$ group of order $n$, and assume that $p$ does not divide the order of $\Aut(E)$.
   Then, the subgroups of $\mathrm{Hol}(N)$ of order $p$ acting on $N$ with trivial stabilisers are 
   $$
		P_i=\left\langle \Big[\begin{pmatrix}
		1\\
		1
		\end{pmatrix},\begin{pmatrix}
		1 & \gamma_i\\
		0&1
		\end{pmatrix}\Big] \right\rangle\subset \Hol(N),\qquad i\in\Z_p.
		$$
 \end{lemma}
\begin{proof}
 Let us look for the order $p$ elements in $\Hol(N)$. Since $p$ does not divide the order of $\Aut(E)$, the equality
		$$
		\Big[\begin{pmatrix}
		m\\
		a
		\end{pmatrix},
		\begin{pmatrix}
		\alpha & \gamma\\
		0&\lambda
		\end{pmatrix}\Big]^p=
		\Big[\begin{pmatrix}0\\
		1
		\end{pmatrix},
		\begin{pmatrix}
		1 & 0\\
		0& 1
		\end{pmatrix}\Big]$$ implies that $\alpha=\lambda=1$. In that case, the equality of vectors gives $a^p=1$, whence $a=1$ because $p$ does not divide the order of $E$. We obtain that the elements of $\Hol(N)$ of order $p$ are
		$$
		\ \Big[\begin{pmatrix}
		m\\
		1
		\end{pmatrix},\begin{pmatrix}
		1 & \gamma_{i}\\
		0&1
		\end{pmatrix}\Big],\quad  i\in\Z_p.
		$$
		If $\sigma=1$ (direct product), all the coboundaries $\gamma_i$ are trivial, 
		and there is a unique subgroup of order $p$:
		$$
		\Big\{\Big[\begin{pmatrix}
		m\\
		1
		\end{pmatrix},\begin{pmatrix}
		1 & 0\\
		0&1
		\end{pmatrix}\Big]\ \colon m\in\Z_p
		\Big\}\subset \Hol(N).$$
Its action on $N$ is given by multiplication on $N$, 
that is $ ([u, Id],v)\to u\cdot v $,
 and it has trivial stabilisers.
 
If $\sigma\ne 1$  the
		previous one is a subgroup of $\Hol(N)$, but there are also $p$ different coboundaries $\gamma_i=i\gamma_1$. 
		By considering 
		$$
		\left\langle \Big[\begin{pmatrix}
		0\\
		1
		\end{pmatrix},\begin{pmatrix}
		1 & \gamma_1\\
		0&1
		\end{pmatrix}\Big] \right\rangle=
		\Big\{\Big[\begin{pmatrix}
		0\\
		1
		\end{pmatrix},\begin{pmatrix}
		1 & \gamma_i\\
		0&1
		\end{pmatrix}\Big]\ \colon i\in\Z_p
		\Big\}\subset \Hol(N)
		$$
		we get a subgroup of order $p$, but all its elements stabilise $(0,1)\in N$. If $m$ is not zero, then $$
		\Big[\begin{pmatrix}
		m\\
		1
		\end{pmatrix},\begin{pmatrix}
		1 & \gamma_j\\
		0&1
		\end{pmatrix}\Big]=[\begin{pmatrix}
		1\\
		1
		\end{pmatrix},\begin{pmatrix}
		1 & \gamma_i\\
		0&1
		\end{pmatrix}\Big]^m \quad \mbox{ for }
		i\equiv m^{-1}j\,(\mathrm{mod}\,p).$$
and we obtain the subgroups 
		$$
		P_i=\left\langle \Big[\begin{pmatrix}
		1\\
		1
		\end{pmatrix},\begin{pmatrix}
		1 & \gamma_i\\
		0&1
		\end{pmatrix}\Big] \right\rangle
  =\left\{
  \Big[\begin{pmatrix}
		m\\
		1
		\end{pmatrix},\begin{pmatrix}
		1 & \gamma_{mi}\\
		0&1
		\end{pmatrix}\Big],\quad m\in \Z_p 
  \right\}
  \subset \Hol(N),\quad i\in\Z_p.
		$$
		The orbit of the trivial element $(0,1)\in N$ under the action of $P_i$ has $p$ elements, namely the elements $(m,1)$. Therefore, $P_i$ acts with trivial stabilisers.
\end{proof}

\begin{remark} The condition $p\nmid |\Aut(E)|$ depeds only on $n$ and not on the group $E$, since it is equivalent to $p\nmid \theta(n)$ where 
	$\theta\colon\mathbb{Z}_{>0}\longrightarrow\mathbb{Z}_{>0}$ is defined for a prime power by
	$$
	\theta(q^m)=\mid\GL_m(\F_q)\mid=(q^m-1)(q^m-q)\dots (q^m-q^{m-1})
	$$
	and extended multiplicatively (see \cite{BH}). For example, $\theta(12)=\theta(3)\theta(4)=2\cdot 3\cdot 2=12.$
\end{remark}

\begin{lemma}\label{conj}
Let $p$ and $n$ be as in the fixed assumption  and assume that $p\nmid \theta(n)$. If $G$ is a regular subgroup of $\Hol(\Z_p\rtimes_{\sigma}E)$, with $E$ group of order $n$, then the $p$-Sylow subgroup of $G$ is either $P_0$ or $P_{-1}$.
Both subgroups are normal subgroups of $\Hol(\Z_p\rtimes_{\sigma}E)$.
 \end{lemma}
 \begin{proof}
Let $P_j$ be the $p$-Sylow subgroup of the regular subgroup $G$. The conjugation with elements 
$$
\Big[\begin{pmatrix}
	0\\
	a
	\end{pmatrix}, \begin{pmatrix}
	\alpha &\gamma_{i}\\
	0&\lambda
	\end{pmatrix}\Big]\in G$$
 gives 
 $$
 \begin{array}{l}
\Big[\begin{pmatrix}
	0\\
	a
	\end{pmatrix}, \begin{pmatrix}
	\alpha &\gamma_{i}\\
	0&\lambda 
	\end{pmatrix}\Big]\Big[\begin{pmatrix}
		1\\
		1
		\end{pmatrix},\begin{pmatrix}
		1 & \gamma_j\\
		0&1
		\end{pmatrix}\Big]
 \Big[\begin{pmatrix}
	-\alpha^{-1}\gamma_{i}(a^{-1})\\
	\lambda^{-1}(a^{-1})
	\end{pmatrix}, \begin{pmatrix}
	\alpha^{-1} &\gamma_{-\alpha^{-1}i}\\
	0&\lambda^{-1} 
	\end{pmatrix}\Big]=\\[10pt]
=\Big[\begin{pmatrix}
	\sigma(a)\alpha\\
	a
	\end{pmatrix}, \begin{pmatrix}
	\alpha &\gamma_{i+\alpha j}\\
	0&\lambda 
	\end{pmatrix}\Big],\Big[\begin{pmatrix}
	-\alpha^{-1}\gamma_{i}(a^{-1})\\
	\lambda^{-1}(a^{-1})
	\end{pmatrix}, \begin{pmatrix}
	\alpha^{-1} &\gamma_{-\alpha^{-1}i}\\
	0&\lambda^{-1} 
	\end{pmatrix}\Big]
=\Big[\begin{pmatrix}
	\sigma(a)\alpha(1+j)-\alpha j\\
	1
	\end{pmatrix}, \begin{pmatrix}
	1 &\gamma_{\alpha j}\\
	0 &1 
	\end{pmatrix}\Big],
\end{array}$$
which is in $P_j$ if and only if $j=0$ or $j=-1$.  
On the other hand, for any element in $\Hol(\Z_p\rtimes_{\sigma}E)$ we have
$$
\Big[\begin{pmatrix}
	m\\
	a
	\end{pmatrix}, \begin{pmatrix}
	\alpha &\gamma_{i}\\
	0&\lambda
	\end{pmatrix}\Big]
 =
 \Big[\begin{pmatrix}
	0\\
	a
	\end{pmatrix}, \begin{pmatrix}
	\alpha &\gamma_{i}\\
	0&\lambda
	\end{pmatrix}\Big]\Big[\begin{pmatrix}
 n
	\\
	1
	\end{pmatrix}, \begin{pmatrix}
	1 &0\\
	0&1
	\end{pmatrix}\Big]
 $$
 and
 $$
\Big[\begin{pmatrix}
	m\\
	a
	\end{pmatrix}, \begin{pmatrix}
	\alpha &\gamma_{i}\\
	0&\lambda
	\end{pmatrix}\Big]
 =
 \Big[\begin{pmatrix}
	0\\
	a
	\end{pmatrix}, \begin{pmatrix}
	\alpha &\gamma_{i+\alpha n}\\
	0&\lambda
	\end{pmatrix}\Big]\Big[\begin{pmatrix}
n
	\\
	1
	\end{pmatrix}, \begin{pmatrix}
	1 &\gamma_{-n}\\
	0&1
	\end{pmatrix}\Big]
 $$
with $n=\sigma(a)^{-1}\alpha^{-1}m$. Therefore, the above computation shows that both $P_0$ and $P_{-1}$ are normal subgroups of the holomorph. 
 \end{proof}

Under our assumptions, the group $\Z_p$, the trivial brace of size $p$, is an ideal of any skew brace of size $np$ and taking the quotient structure we obtain a skew brace of size $n$. We analyze this quotient structure in terms of regular subgroups of the corresponding holomorphs.
 
	\begin{theorem}\label{npgivesn} Let $p$ and $n$ be as in the assumption. Let $(B,\cdot,\circ)$ be a skew brace of size $np$ with additive group $N=\Z_p\rtimes_{\sigma} E$ and multiplicative group $G$, where $(E,\cdot)$ is a group of order $n$ and $\sigma: E\to \Z_p^*=\Aut(\Z_p)$ is a group homomorphism. If $p$ does not divide the order of $\Aut(E)$, then 
		the quotient skew brace $B/\Z_p$ is a skew brace $(B_n, \cdot,\circ)$ with additive group $(E,\cdot)$ such that $\sigma$ is a skew brace homomorphism $B_n\to \Z_p^*$ and $G\simeq \Z_p\rtimes_{\tau} (B_n,\circ)$ for some $\tau: (B_n,\circ)\to \Z_p^*$ group homomorphism. 
		
		An isomorphism class of skew braces of size $np$ with additive group $\Z_p\rtimes_{\sigma} E$ provides an isomorphism class of skew braces of size $n$ with additive group $E$.
	\end{theorem}
	
\begin{proof}
As in Proposition \ref{GV}, the multiplicative group of $B$ is the regular subgroup 
$$
G=\Big\{\Big[\begin{pmatrix}
		m\\
		a
		\end{pmatrix},\Lambda_{(m,a)}\Big]: m\in\Z_p,\ a \in E\Big\}\subseteq \Hol(\Z_p\rtimes_{\sigma} E).
        $$
Let $P_i$, $i\in\{0,-1\}$, be  the $p$-Sylow subgroup of $G$. Since the projection of $P_i$ on its first component equals the subgroup of order $p$ of $N$, $P_i$ corresponds to a subbrace of $B$. Since $\Z_p$ is a characteristic subgroup of $N$, this subbrace is an ideal and the corresponding quotient is a brace of size $n$.

Using the assumptions, Schur-Zassenhaus theorem states that $P_i$ has a complement $\hat F$, a subgroup of order $n$, and that $G=P_i\rtimes \hat F$, the inner semidirect product, with $\hat F$ acting on $P_i$ by conjugation. Due to regularity of $G$ and taking into account the description of elements in $P_i$ and elements in $\Hol(N)$, we must have
\begin{equation}\label{F}		
\hat F=\Big\{ h_a\coloneqq\Big[\begin{pmatrix}
		m_a\\
		a
		\end{pmatrix},\begin{pmatrix}
		\alpha_a & \gamma_{i_a}\\
		0&\lambda_a
		\end{pmatrix}\Big] \, : \, a \in E\Big\},
\end{equation}
where $a \mapsto m_a$, $a \mapsto \alpha_a$, $a \mapsto \gamma_{i_a}$, $a \mapsto \lambda_a$ define maps from $E$ to $\Z_p$, $\Z_p^*,$ $C^1(E,\Z_p)$ and $\Aut E$, respectively. The product of two elements  $h_a,h_b\in \hat F$ is
		$$h_ah_b=\Big[\begin{pmatrix}
		m_a+\sigma(a)(\alpha_a m_b+\gamma_{i_a}(b))\\
		a\lambda_a(b)
		\end{pmatrix}, \begin{pmatrix}
		\alpha_a\alpha_b & \alpha_a\gamma_{i_b}+\gamma_{i_a}\\
		0&\lambda_a\lambda_b
		\end{pmatrix}\Big].
		$$ Note that we have used the equality $\gamma_{i_a}\lambda_b=\gamma_{i_a}$, which derives from  $\sigma\lambda_b=\sigma$.
Since $\hat F$ is a subgroup, we have $h_ah_b=h_{a\lambda_a(b)}$. This gives rise to the following equalities:
\begin{align}
m_{a\lambda_a(b)}&=m_a+\sigma(a)\alpha_am_b+\sigma(a)\gamma_{i_a}(b),\label{sgc1}\\
\alpha_{a\lambda_a(b)}&=\alpha_a\alpha_b,\\
\gamma_{i_{a\lambda_a(b)}}&=\gamma_{i_a}+\alpha_a\gamma_{i_b}, \label{sgc3} \\
\lambda_{a\lambda_a(b)}&=\lambda_a\lambda_b. \label{sgc4}
\end{align}
From \ref{sgc4}, we obtain that  $F=\left\{
		(a, \lambda_a) :\, a\in E
		\right\}
		$
		is a regular subgroup of $\Hol(E)$, and $h_a\to (a,\lambda_a)$ gives $\hat F\simeq F$. 
        Therefore, $G$ is isomorphic to an outer semidirect product $\Z_p\rtimes_{\tau} F$ for 
        $\tau:F\simeq \hat F\to \Aut(P_i)\simeq \Aut(\Z_p)\simeq \Z_p^*$.
        Since
		$a\circ b=a\cdot \lambda_a(b)$ and $\sigma\lambda_a=\sigma$ we see that $\sigma:E\to \Z_p^*$ is in fact a skew brace homomorphism and $G$ is as in the statement.
  
Let us prove the last sentence. If we consider  $
		M=\begin{pmatrix}
		\alpha & \gamma\\
		0&\lambda
		\end{pmatrix}
		\in\Aut(\Z_p\rtimes_{\sigma} E)\subseteq \Hol(\Z_p\rtimes_{\sigma} E)$
	and the inner automorphism $\Phi_M$ of conjugation by $M$ in $\Hol(\Z_p\rtimes_{\sigma}E)$, then
		$$
		\Phi_M \Big(\Big[\begin{pmatrix}
		m\\
		a
		\end{pmatrix},\begin{pmatrix}
		\alpha_a & \gamma_{i_a}\\
		0&\lambda_a
		\end{pmatrix}\Big]\Big)=\Big[M\begin{pmatrix}
		m\\
		a
		\end{pmatrix},M\begin{pmatrix}
		\alpha_a & \gamma_{i_a}\\
		0&\lambda_a
		\end{pmatrix}M^{-1} \Big]=\Big[\begin{pmatrix}
		\alpha m+\gamma(a)\\
		\lambda(a)
		\end{pmatrix}, \begin{pmatrix}
		\alpha_a & *\\
		0&\lambda\lambda_a\lambda^{-1}
		\end{pmatrix}\Big].
		$$
	Therefore, the subgroup of $\Hol(E)$ provided by a conjugate of $G$ is 
		$$
		\Phi_{\lambda}(F)=\left\{
		(\lambda(a), \lambda\lambda_a\lambda^{-1}) :\quad a\in E
		\right\}
		$$
        where $\Phi_{\lambda}$ is the inner automorphism of $\Hol(E)$ of conjugation by $\lambda\in\Aut(E)\subseteq \Hol(E).$
\end{proof}
	
Now, the remaining question is to determine how many skew braces of size $np$ give the same quotient structure of size $n$.
We prove that we have just the two described in Lemma \ref{rm}.
	
\begin{theorem}\label{2Gs}
Let $p$ and $n$ be as in the assumption and such that $p\nmid \theta(n)$. Let $B_n$ be a skew brace of size $n$ with additive group $(E,\cdot)$ and let
	$F=\{(a,\lambda_a): a\in E\}$ the corresponding regular subgroup of $\Hol(E)$.
		
	For every $\sigma\in\Hom(B_n,\Z_p^*)$ and every $\tau\in \Hom(F,\Z_p^*)$,
	$$
	G_{(\sigma,\tau)}=\left\{\Big[\begin{pmatrix}m\\a\end{pmatrix},
	\begin{pmatrix}\sigma(a)^{-1}\tau(a,\lambda_a) & 0\\0&\ \lambda_a\end{pmatrix}\Big]
		\ \colon 
		m\in\Z_p,\  a\in E
		\right\}
	$$
	$$G'_{(\sigma,\tau)}=\left\{\Big[\begin{pmatrix}m\\a\end{pmatrix},
		\begin{pmatrix}\tau(a,\lambda_a) & \gamma_{-\sigma(a^{-1})m}\\0&\lambda_a\end{pmatrix}\Big]
		\ \colon 
		m\in\mathbb Z_p,\  a\in E
		\right\}
	$$
	are regular subgroups of $\Hol(\Z_p\rtimes_{\sigma}E)$
	isomorphic to $\Z_p\rtimes_{\tau} F$.
		
	We have $G_{(1,\tau)}=G'_{(1,\tau)}$ for any $\tau$. For a
	non-trivial $\sigma$,
	$G_{(\sigma,\tau)}$ and $G'_{(\sigma,\tau)}$ are not conjugate in $\Hol(\Z_p\rtimes_{\sigma}E)$. Moreover, $G_{(\sigma,\tau)}$
	and $G'_{(\sigma,\tau)}$ represent the unique conjugacy classes of regular 
	subgroups of $\Hol(\Z_p\rtimes_{\sigma} E)$ among the isomorphism class of $\Z_p\rtimes_{\tau}F$ such that the quotient of the corresponding skew brace of size $np$ by the trivial brace $\Z_p$ is isomorphic to $B_n$.
\end{theorem}
	
\begin{proof}
Since $\sigma$ gives a group homomorphism $E\to \Z_p^*$, we can consider the semidirect product $N=\Z_p\rtimes_{\sigma}E$.  Let $G$ be a regular subgroup of $\Hol(N)$ and $B$ the corresponding skew brace of size $np$. Let us assume that the quotient structure $B/\Z_p$ corresponds to $F$, so that $G\simeq \Z_p\rtimes_{\tau} F$. Following the description of $G$ from the proof of Theorem \ref{npgivesn}, either $G=P_0 \rtimes \hat F$ or $G=P_{-1}\rtimes \hat F$, for some subgroup $\hat F$ as in \eqref{F}. The action given by $F\simeq \hat F\to \Aut(P_i)\simeq \Aut(\Z_p)\simeq \Z_p^*$ is 
$$
(a,\lambda_a)\to h_a\to \Phi_{h_a} \in\Aut(P_i)
$$
followed by the identification of the element in $\Z_p*$ corresponding to the conjugation in $P_i$ by $h_a$.
We obtain 
 $(a,\lambda_a) \mapsto \sigma(a)\alpha_a\in\Z_p^*$ for $i=0$ and $(a,\lambda_a) \mapsto \alpha_a\in \Z_p^*$ for $i=-1$. We want this to agree with the given $\tau:F\to \Z_p^*.$

Let us see that, up to conjugation in $\Hol(E)$, we have actually $G=P_i \rtimes F_0$, where
$$F_0=\Big\{\Big[\begin{pmatrix}
		0\\
		a
		\end{pmatrix},\begin{pmatrix}
		\alpha_a & 0\\
		0&\lambda_a
		\end{pmatrix}\Big] \, : \, a \in E\Big\}\subseteq \Hol(E),$$
with $\lambda_a$ determined by $F$, $\alpha_a=\sigma(a^{-1})\tau(a,\lambda_a)$ for $i=0$ and $\alpha_a=\tau(a,\lambda_a)$ for $i=-1$.

 The equality \eqref{sgc3} is equivalent to $i_{a\lambda_a(b)}=\alpha_a i_b+i_a$, which gives that $a \to i_a$ is a 1-cocycle from $F$ to $\Z_p$ with respect to the action given by $\alpha\colon a\to\alpha_a$, and therefore
 a coboundary. We have then $i_a=j(1-\alpha_a)$ for some fixed $j$. Let us consider first a subgroup $P_0\rtimes \hat F$ with $\hat F$ as in \eqref{F}. We have
		$$\begin{array}{l}
		\Big[\begin{pmatrix}
		0\\
		1
		\end{pmatrix},\begin{pmatrix}
		1 & \gamma_j\\
		0&1
		\end{pmatrix}\Big]\Big[\begin{pmatrix}
		0\\
		a
		\end{pmatrix},\begin{pmatrix}
		\alpha_a & 0\\
		0&\lambda_a
		\end{pmatrix}\Big]
\Big[\begin{pmatrix}
		0\\
		1
		\end{pmatrix},\begin{pmatrix}
		1 & -\gamma_j\\
		0&1
		\end{pmatrix}\Big]\\[12pt]
=\Big[\begin{pmatrix}
		\gamma_j(a)\\
		a
		\end{pmatrix},\begin{pmatrix}
		\alpha_a & -\alpha_a\gamma_j+\gamma_j\\
		0&\lambda_a
		\end{pmatrix}\Big]
=\Big[\begin{pmatrix}
		\gamma_j(a)\\
		a
		\end{pmatrix},\begin{pmatrix}
		\alpha_a & \gamma_{i_a}\\
		0&\lambda_a
		\end{pmatrix}\Big],
\end{array}
		$$
using $-\alpha_a\gamma_j+\gamma_j=(1-\alpha_a)\gamma_j=\gamma_{i_a}$.

On the other hand, inside $P_0\rtimes \hat F$, the subgroup $F$ is determined up to conjugation by $P_0$. Since both $m_a$ and $\gamma_j(a)$ satisfy \eqref{sgc1}, we obtain that $a\mapsto m_a-\gamma_j(a)$ is a 1-cocycle with respect to the action of $F$ on $\Z_p$ given by $\sigma(a)\alpha_a$ and therefore a coboundary:  $m_a-\gamma_j(a)=k(1-\sigma(a)\alpha_a)$ for some fixed integer $k$. By conjugation inside $P_0\rtimes \hat F$, we obtain
		$$\begin{array}{l}
		\Big[\begin{pmatrix}
		k\\
		1
		\end{pmatrix},\begin{pmatrix}
		1 & 0\\
		0&1
		\end{pmatrix}\Big]\Big[\begin{pmatrix}
		\gamma_j(a)\\
		a
		\end{pmatrix},\begin{pmatrix}
		\alpha_a & \gamma_{i_a}\\
		0&\lambda_a
		\end{pmatrix}\Big]
\Big[\begin{pmatrix}
		-k\\
		1
		\end{pmatrix},\begin{pmatrix}
		1 & 0\\
		0&1
		\end{pmatrix}\Big]\\[12pt]
=\Big[\begin{pmatrix}
		k+\gamma_j(a)-\sigma(a)\alpha_a k\\
		a
		\end{pmatrix},\begin{pmatrix}
	\alpha_a & \gamma_{i_a}\\
		0&\lambda_a
		\end{pmatrix}\Big]
=\Big[\begin{pmatrix}
		m_a\\
		a
		\end{pmatrix},\begin{pmatrix}
	\alpha_a & \gamma_{i_a}\\
		0&\lambda_a
		\end{pmatrix}\Big].
        \end{array}
		$$
To sum up, we have obtained that, given any $\hat F$ as in \eqref{F},  $P_0\rtimes F_0$ is conjugate to $P_0\rtimes \hat F$ inside $\Hol(N)$.

Let us consider now  $P_{-1}\rtimes \hat F$ with $\hat F$ as in \eqref{F}. For any integer $k$, we have
		$$\begin{array}{l}
		\Big[\begin{pmatrix}
		0\\
		1
		\end{pmatrix},\begin{pmatrix}
		1 & \gamma_{j+k}\\
		0&1
		\end{pmatrix}\Big]\Big[\begin{pmatrix}
		0\\
		a
		\end{pmatrix},\begin{pmatrix}
		\alpha_a & 0\\
		0&\lambda_a
		\end{pmatrix}\Big]
\Big[\begin{pmatrix}
		0\\
		1
		\end{pmatrix},\begin{pmatrix}
		1 & -\gamma_{j+k}\\
		0&1
		\end{pmatrix}\Big]\\[12pt]
=\Big[\begin{pmatrix}
		\gamma_{j+k}(a)\\
		a
		\end{pmatrix},\begin{pmatrix}
		\alpha_a & -\alpha_a\gamma_{j+k}+\gamma_{j+k}\\
		0&\lambda_a
		\end{pmatrix}\Big]
=\Big[\begin{pmatrix}
		\gamma_{j+k}(a)\\
		a
		\end{pmatrix},\begin{pmatrix}
		\alpha_a & \gamma_{i_a}+(1-\alpha_a)\gamma_k\\
		0&\lambda_a
		\end{pmatrix}\Big].
\end{array}
		$$
As above, $m_a-\gamma_j(a)=k(1-\sigma(a)\alpha_a)$ for some fixed integer $k$. By conjugation inside $P_{-1} \rtimes \hat F$, we obtain
		$$\begin{array}{l}
		\Big[\begin{pmatrix}
		k\\
		1
		\end{pmatrix},\begin{pmatrix}
		1 & \gamma_{-k}\\
		0&1
		\end{pmatrix}\Big]\Big[\begin{pmatrix}
		\gamma_{j+k}(a)\\
		a
		\end{pmatrix},\begin{pmatrix}
		\alpha_a & \gamma_{i_a}+(1-\alpha_a)\gamma_k\\
		0&\lambda_a
		\end{pmatrix}\Big]
\Big[\begin{pmatrix}
		-k\\
		1
		\end{pmatrix},\begin{pmatrix}
		1 & \gamma_k\\
		0&1
		\end{pmatrix}\Big]\\[12pt]
=\Big[\begin{pmatrix}
		k+\gamma_j(a)-\sigma(a)\alpha_a k\\
		a
		\end{pmatrix},\begin{pmatrix}
	\alpha_a & \gamma_{i_a}\\
		0&\lambda_a
		\end{pmatrix}\Big]
=\Big[\begin{pmatrix}
		m_a\\
		a
		\end{pmatrix},\begin{pmatrix}
	\alpha_a & \gamma_{i_a}\\
		0&\lambda_a
		\end{pmatrix}\Big].
        \end{array}
		$$
and we see that, given any $\hat F$ as in \eqref{F},  $P_{-1}\rtimes F_0$ is conjugate to $P_{-1}\rtimes \hat F$ inside $\Hol(N)$.

In conclusion, we have seen that $G_{(\sigma,\tau)}=P_0\rtimes F_0$ and $G'_{(\sigma,\tau)}=P_{-1}\rtimes F_0$ are, up to conjugation, the unique regular subgroups of $\Hol(N)$ giving $B_n$ by the quotient procedure.
\end{proof}

	\begin{remark} The twofold semidirect product of skew braces is a particular case of the semidirect product of digroups as defined by Facchini and Pompili \cite{FP}. Indeed, if we are under the hypotheses of Proposition \ref{dsdp} and in \cite[Proposition 4.2]{FP} we specify $\varphi_{\star}=\sigma$, $\varphi_{\circ}=\tau$ and $\Lambda=\mathrm{Id}$, we obtain a brace isomorphic to the skew brace corresponding to $G_{(\sigma,\tau)}$ in Theorem \ref{2Gs}. On the other hand, if we keep the same choice for $\varphi_{\star}$ and $\varphi_{\circ}$ but we take $\Lambda=\sigma$, we obtain a brace isomorphic to the skew brace corresponding to $G'_{(\sigma,\tau)}$. 
	\end{remark}
	
	Once this 2-to-1 correspondence (for non-trivial $\sigma$) is established, in order to count different skew brace structures, we have to determine conjugation orbits in the families
	$
	\{G_{(\sigma,\tau)}\}_{\sigma,\tau},$ and $ \{G'_{(\sigma,\tau)}\}_{\sigma,\tau}
	$
	or equivalently, isomorphism classes of skew brace structures.
	Therefore, we should take $\sigma$ up to skew brace automorphisms of $(B_n,\cdot,\circ)$ and
	$\tau$ up to group automorphisms of $(B_n, \circ)$.

	\subsection{Algorithm}
	
	Our aim is to determine all skew braces of size $np$ from skew braces of size $n$. Therefore we need the following computation:
	\begin{description}
		\item[Step 0] Determine the isomorphism classes of groups $E$ of order $n$ and the number of skew braces of size $n$ of each type $(E,F)$, identifying $F$ with a representative for a conjugacy class of regular subgroups of $\Hol(E)$.
		
		For every $E$  compute the stabilisers $\Sigma_{\sigma}$ for the action
		$$
		\begin{array}{ccc}
		\Aut(E)\times \Hom(E,\Z_p^*)&\longrightarrow&  \Hom(E,\Z_p^*) \\
		(g, \sigma) &\longrightarrow & \sigma g
		\end{array}
		$$
		namely, $\Sigma_{\sigma}=\{g\in\Aut(E): \  \sigma g=\sigma\}$.
		
	\end{description}
	
	Next, we consider as an input a pair $(E,F)$, with $(E,\cdot)$ a group of order $n$ and $F$ a regular subgroup of $\Hol(E)$. 
    That is, the input is a skew brace $B_n$ of size $n$.  The steps
	to count all the skew braces of size $np$ having additive group isomorphic to $\Z_p\rtimes E$ and multiplicative group isomorphic to $\Z_p\rtimes F$ are the following:
	
	\begin{description}
		
		\item[Step 1] Determine
		$
		\Hom(B_n,\Z_p^*)=\{\sigma\in \Hom(E,\Z_p^*) : \  \pi_2(F)\subseteq \Sigma_{\sigma}\ \}
		$, where $\pi_2$ is the projection of $\Hol(E)=E\rtimes \Aut(E)$ onto the second component $\Aut(E)$.
		
		\item[Step 2] Determine the skew brace automorphism group
		$
		\Aut(B_n)=\{g\in\Aut(E) : \  \Phi_g(F)=F\ \}
		$
        where $\Phi_g$ is the inner automorphism of $\Hol(E)$ corresponding to the element $g\in\Aut(E)\subseteq \Hol(E)$.
		
		\item[Step 3] Compute the number of orbits of the action
		$$
		\begin{array}{ccc}
		\Aut(B_n) \times \Hom(B_n,\Z_p^*)  &\longrightarrow&  \Hom(B_n,\Z_p^*). \\
		(g, \sigma) &\longrightarrow & \sigma g
		\end{array}
		$$
		
		Notice that the orbit of $\sigma=1$, which corresponds to the direct product,  has a single element. The remaining orbits will give rise to two different skew braces.
		
		The number of additive structures for $F$ is the number of orbits.
		Since orbits preserve the order and the isomorphism class of the kernel, we can keep track of the number of structures according to these parameters.

These orbits give a partition of $\Hom(B_n,\Z_p^*)$ and we consider a system of representatives of the corresponding equivalence classes, which amounts to the different isomorphism classes of semidirect products $\Z_p\rtimes E$.
        
		\item[Step 4] For each $\sigma$ in a system of representatives as above, compute the orbits of the action
		$$
		\begin{array}{ccc}
		( \Aut(B_n)\cap \Sigma_{\sigma}) \times \Hom(F,\Z_p^*) &\longrightarrow&  \Hom(F,\Z_p^*). \\
		(g, \tau) &\longrightarrow & \tau\Phi_g
		\end{array}
		$$
        Notice that  $\Aut(B_n)\cap \Sigma_{\sigma}$ is a subgroup of the normalizer of $\pi_2(F)$ in
		$\Sigma_{\sigma}$.

        These orbits correspond to the isomorphism classes of semidirect products $\Z_p\rtimes F$.
		The number of multiplicative structures  is twice the number of orbits except for $\sigma$ trivial, when we get a single one.
		Since orbits preserve order and isomorphism class of kernel, we can keep track of the number of structures according to these parameters.
	\end{description}
	
	\subsection{Number of skew braces of size $np$}
	
	Taking into account the description given in terms of orbits, we can use the orbit-counting to give formulas for the number of skew braces of size $np$ in terms of the number of skew braces of size $n$.
    
    For a fixed skew brace of size $n$, the number of additive structures for skew braces of size $np$ with quotient $B_n$ is the number of orbits of the action of $\Aut(B_n)$ on $\Hom(B_n, Z_p^*)$. 
    For a given $\sigma\in \Hom(B_n, Z_p^*)$, we have to count 
    the orbits of the action of 
    $\Aut(B_n)\cap \Sigma_{\sigma}$ in $\Hom((B_n,\circ), Z_p^*)$.
    The trivial $\sigma$ gives one skew brace of size $np$ for each orbit, while a nontrivial $\sigma$ gives two skew braces of size $np$ for each orbit.
	
	\begin{theorem}
		Let $p$ be a prime and $n$ a positive integer not divisible by $p$ and such that every group of order $np$ has a normal subgroup of order $p$ and such that $p\nmid \theta(n)$.
        
	  For a skew brace $B_n$ of size $n$, the number of additive structures of skew braces of size $np$ is the number of orbits $|\Hom(B_n,\Z_p^*)/\Aut(B_n)|$ of the right action $(g,\sigma)\to \sigma g$. Therefore, the total amount of additive structures is 
		$$
		\sum_{B_n}   \dfrac{1}{|\Aut(B_n)|}\sum_{\sigma\in \Hom(B_n,\Z_p^*)}|\Stab_{\Aut(B_n)}(\sigma)|,
		$$
		where $B_n$ runs over the isomorphism classes of skew braces of size $n$. 
        
         Let us consider the action of $\Aut(B_n,\cdot) $ on $\Hom((B_n,\cdot),\Z_p^*)$ as above and let $\Sigma_{B_n,\sigma}$ be the stabilizer of $\sigma\in \Hom(B_n,\Z_p^*)$. Let $A_{B_n,\sigma}=\Aut(B_n)\cap \Sigma_{B_n,\sigma}$. Then, the number of multiplicative structures of size $np$ corresponding to the additive structure defined by $\sigma$ is obtained from the orbits of the action $(g,\tau)\to \tau\Phi_g$ of $A_{B_n,\sigma}$ on $\Hom((B_n,\circ),\Z_p^*)$.
        The total amount of isomorphism classes of skew braces of size $np$ is
		$$
		\sum_{B_n}\,
		\left(
		\dfrac{1}{|A_{B_n,1}|}\sum_{\tau\in \Hom((B_n,\circ),\Z_p^*)}
		|\Stab_{A_{B_n,1}}(\tau)|
		+2\sum_{\sigma\ne 1}
		\dfrac{1}{|A_{B_n,\sigma}|}\sum_{\tau\in \Hom((B_n,\circ),\Z_p^*)}
		|\Stab_{A_{B_n,\sigma}}(\tau)|\right)
		$$
		where $\sigma$ runs over a system of representatives 
        for the orbits $|\Hom(B_n,\Z_p^*)/\Aut(B_n)|$.
        The case $\sigma=1$, corresponds to the semidirect product of $B_n$ and the trivial brace of order $p$. 
	\end{theorem}		

\begin{remark}
 Using \textit{Burnside lemma}, the number of orbits can also be expressed in terms of fixed points. Therefore, we can give the number of additive structures for skew braces of size $np$ as 
 $$
		\sum_{B_n}\,   \dfrac{1}{|\Aut(B_n)|}\sum_{g \in \Aut(B_n)}|\operatorname {Fix}_{\Hom(B_n,\Z_p^*)}(g)|
		$$
        where $\operatorname{Fix}$ denotes the number of fixed points of the element acting on the set. 
The total amount of skew braces of size $np$ is
$$\sum_{B_n}\,
		\left(
		\dfrac{1}{|A_{B_n,1}|}\sum_{g \in A_{B_n,1}}|\operatorname {Fix}_{\Hom((B_n,\circ),\Z_p^*)}(g)|       
		+2\sum_{\sigma\ne 1}
		\dfrac{1}{|A_{B_n,\sigma}|}\sum_{ g \in A_{B_n,\sigma}}|\operatorname {Fix}_{\Hom((B_n,\circ),\Z_p^*)}(g)|  \right)  
		$$
\end{remark}

	\section{Groups of order 12 and skew braces of size~12}\label{twelve}
	
	There are five isomorphism classes of  groups of order 12, two abelian ones represented by $C_{12}$ and $C_6\times C_2$ and three non-abelian ones, represented by the alternating group $A_4$, the dihedral group $D_{2\cdot 6}$ and the dicyclic group $\Dic_{12}$. By computation with Magma \cite{BCP} of conjugacy classes of regular subgroups of the holomorph, we obtain that the number of skew braces with additive group $E$ and multiplicative group $F$ is as shown in the following table.
	\begin{center}
		\begin{tabular}{|c||c|c|c|c|c|}
			\hline
			$E \backslash F$ & \, \, $C_{12}$ \, \, & $C_6\times C_2$ & \, \, $A_4$ \, \, \, & \, \, $D_{2\cdot 6}$ \, \, & \,  $\Dic_{12}$ \, \, \\
			\hline
			\hline
			$C_{12}$ &1&1&0&2& 1\\
			\hline
			$C_6\times C_2$ &1&1&1&1& 1\\
			\hline
			$A_4$ &0&2&4&0& 2\\
			\hline
			$D_{2\cdot 6}$ &2&2&0&4& 2\\
			\hline
			$\Dic_{12}$ &2&2&0&4& 2\\
			\hline
		\end{tabular}
	\end{center}
	
	We determine now for each group $E$ of order 12, its automorphism group $\Aut E$, the group $\Hom(E,\Z_p^*)$ and the stabilizers of the corresponding group action.
	
	Since $\Z_p^*=\langle \zeta \rangle $ is a cyclic group of order $p-1$,
	for each $d\mid p-1$ it has a unique subgroup $\mu_d=\langle \zeta^{\frac{p-1}d} \rangle$ of order $d$. We  write $\zeta_d$ for  $ \zeta^{\frac{p-1}d}$.
	In order to have elements in
	$\Hom(E ,\Z_p^*)$ with image in $\mu_d$ it is necessary that
	$d\mid D\coloneqq\gcd (\exp(E),p-1)$, where $\exp(E)$ is the exponent of the group $E$.
	
	\subsection{$E=C_{12}$}\label{ciclic}

\begin{center}
    \begin{tabular}{|c|c|c|}
\multicolumn{3}{c}{$E=\langle c\rangle\simeq C_{12}$}\\[1ex]
\hline
$\Aut(E)$&$\langle g_5\rangle\times \langle g_7\rangle\simeq\Z_{12}^*\simeq \Z_2\times \Z_2$ &
	 $g_i(c)=c^i$\\[1ex]
$\Hom(E,\Z_p^*)$ & 
$\mu_{D}=\langle \zeta_D\rangle$, $D=\gcd(12,p-1)$ &
$\sigma \to \sigma(c)=\zeta_d^j$, $\gcd(j,d)=1$ $\rightsquigarrow (d,j\bmod d)$\\
\hline
\end{tabular}
  \end{center}  
	
    Let $\sigma\in\Hom(E,\Z_p^*)$ and let $d$ be its order.
    Then $\sigma$ can be identified as a pair $(d, j)$ for $j\in(\Z/d\Z)^*$, so that $(d,j)g_i=(d,ij)$. 
    Elements of the same order are in the same orbit under the action of $\Aut(E)$ since
	$
	(d,j)g_i=(d,j') \mbox{ for } i\equiv j'j^{-1}\,(\mathrm{mod}\,d).
	$
	We use the size  $k=12/d$  of the kernel of $\sigma$ to parametrize the orbits. As for the stabilizers, since
	$\sigma g_i=\sigma\iff \sigma(c)^i=\sigma(c)$,
	we have
	$
	\Sigma_{\sigma}=\{g_i\in \Aut(E) : i\equiv 1 \bmod d\},
	$
	which identifies as a subgroup of $\Z_{12}^*=\{1,5,7,11\}$.
	\begin{center}
		\begin{tabular}{c|l | c}
			$k$  &Orbit of $\sigma=(d,j\bmod d)$& $\Sigma_{\sigma}$ \\
			\hline
			$12$&  $(1,1)$ & $\Aut(E)$\\
			$6$&  $(2,1)$ & $\Aut(E)$\\
			$4$&   $(3,1)\xrightarrow{g_5}(3,2)$& $\{1,g_7\}$\\
			$3$&  $(4,1)\xrightarrow{g_7}(4,3)$& $\{1,g_5\}$\\
			$2$&   $(6,1)\xrightarrow{g_5}(6,5)$& $\{1,g_7\}$\\
			$1$&   
			$(12,1)\xrightarrow{g_5}(12,5)\xrightarrow{g_{11}}(12,7)\xrightarrow{g_5}(12,11)$ & $\{1\}$\\
		\end{tabular}
	\end{center}

	The number of isomorphism classes of semidirect products
	$\Z_p\rtimes C_{12}$ is equal to the number of divisors of
	$D=\gcd(12,p-1)$
	and can be labelled by the size $k$ of $\Ker\sigma$ (size 12 means direct product).

	\subsection{$E=C_{6}\times C_2$}\label{C62}

    \begin{center}
    \begin{tabular}{|c|c|c|}
\multicolumn{3}{c}{$E=\langle a\rangle \times \langle b \rangle\simeq C_{6}\times C_2$}\\[1ex]
\hline
$\Aut(E)$ &  $\langle g_1,\ g_2\rangle\simeq D_{2\cdot 6}$ &
	 $g_1(a)=ab,g_1(b)=b,g_2(a)=a^2b,g_2(b)=a^3b$\\[1ex]
$\Hom(E,\Z_p^*)$ & 
$\mu_{D}\times \mu_2=\langle \zeta_D\rangle\times \langle -1\rangle $, $D=\gcd(6,p-1)$ &
$\sigma(a)=\zeta_d^j$, $\sigma(b)=(-1)^i$  $\rightsquigarrow (j\bmod d,i\bmod 2)$\\
\hline
\end{tabular}
  \end{center}

	We get the following orbits and stabilizers for the action of $\Aut(E)$. We give stabilizers of elements of each orbit as conjugates of the stabilizer of the first one and the isomorphism class of the group.
	\begin{center}
		\begin{tabular}{c|c|l| lll}
			$k$& $d$ & Orbit of $\sigma=(j\bmod d,i\bmod2)$ 
            & $\Sigma_{\sigma}$ && \\
			\hline
			12&1    &  $(1,0)$ &$\Aut(E)$ &&\\
			6&2    & $(1,0)\xrightarrow{g_2} (0,1)\xrightarrow{g_2} (1,1)$  
            &
            $g_2^m \langle g_2^3, g_1\rangle g_2^{-m}$ & $m=0,1,2$  &$C_2\times C_2$
            \\
			4&3    & $(1,0)\xrightarrow{g_2} (2,0)$ 
            &$g_2^m \langle g_2^2, g_1\rangle g_2^{-m}$&$m=0,1$ &$S_3$ \\
			2&6    & $(1,0)\xrightarrow{g_2} (2,1)\xrightarrow{g_2} (1,1)
			\xrightarrow{g_2} (5,0)\xrightarrow{g_2} (4,1)\xrightarrow{g_2} (5,1)$
            &$g_2^m \langle g_1\rangle g_2^{-m}$&$ m=0,1,2$ & $\C_2$
            \\
		\end{tabular}
	\end{center}
The number of isomorphism classes of semidirect products
	$\Z_p\rtimes ( C_{6}\times C_2)$ is equal to the number of divisors $d$ of
	$D=\gcd(6,p-1)$ 
	and can be labelled by the size $k=12/d$ of $\Ker\sigma$.

	\subsection{$E=A_4$}\label{A4}

    \begin{center}
    \begin{tabular}{|c|c|c|}
\multicolumn{3}{c}{$E=A_4=\langle (1,2,3), (1,2)(3,4)\rangle\subseteq S_4 $}\\[1ex]
\hline
$\Aut(E) $&$S_4$ &
	$
	\phi_s(1,2,3)=(s(1),s(2),s(3))$\\
    & &  $\phi_s(1,2)(3,4)=(s(1),s(2)),(s(3),s(4))$\\[1ex]
$\Hom(E,\Z_p^*)$ & 1 if $p\not\equiv 1\bmod 3$ & \\
&$\langle \zeta_3\rangle$ if $p\equiv 1\bmod 3$ &
$\sigma \to \sigma(1,2,3)=\zeta_3^j \rightsquigarrow j\bmod 3$\\
\hline
\end{tabular}
  \end{center} 

    The information about orbits and stabilizers for the action of $\Aut(A_4)$ on $\Hom(A_4,\Z_p^*)$ is provided in the following table:
	\begin{center}
		\begin{tabular}{c|c|l|l}
			$k$ & $ d$ & Orbit of $\sigma=j \bmod 3$ & Stabiliser $\Sigma_{\sigma}$ \\
			\hline
			12 &1 & $0$   & $S_4$\\
			4&3 & $1\xrightarrow{\phi_{(1,2)}}2$ & $A_4$\\
		\end{tabular}
	\end{center}

	\subsection{$E=D_{2\cdot 6}$}\label{D26}

      \begin{center}
    \begin{tabular}{|c|c|c|}
\multicolumn{3}{c}{$E=\langle r,s \mid r^6=\Id,\ s^2=\Id,\ srs=r^5 \rangle$}\\[1ex]
\hline
$\Aut(E) $&$\langle g_1,\ g_2\rangle\simeq D_{2\cdot 6}$ &
	 $g_1(r)=r^{-1},g_1(s)=s,g_2(r)=r,g_2(s)=rs$\\[1ex] 
$\Hom(E,\Z_p^*)$ & $\mu_{2}\times \mu_2=\{\pm 1 \} \times \{\pm 1 \}$& $\sigma \to  (\sigma(r),\sigma(s))$\\
\hline
\end{tabular}
  \end{center} 
	
	This group $E$ has exponent $6$ and the unique normal subgroups with cyclic quotient are those of index 2. As elements in $\Hom(E,\Z_p^*)$, the kernel of 
	$(1,-1)$ is $\langle r\rangle$ while $(-1,-1)$ and $(-1,1)$ have kernel isomorphic to dihedral group $D_{2\cdot 3}$. Therefore, they give at least two orbits under the action of $\Aut(E)$.
    In fact, this action  can be described as follows:
	\begin{center}
		\begin{tabular}{c|c|l|l}
			$k$ & $ d$ & Orbit of $\sigma$& Stabilizer $\Sigma_{\sigma}$ \\
			\hline
			12 &1 & $(1,1)$   & $\Aut(E)$\\
			6&2 & $(1,-1)$   & $\Aut(E)$\\
			6&2 & $(-1,-1)\xrightarrow{g_2}(-1,1)$   & $\langle g_1,g_2^2\rangle\simeq D_{2\cdot 3}$\\
		\end{tabular}
	\end{center}

     We have three orbits for the action, which means three isomorphism classes of semidirect (or direct) products
	$\Z_p\rtimes E$, parameterized by the isomorphism class of the kernel. For $\sigma=(1,-1)$, we write $\mathbb{Z}_p\rtimes_6^c E$ (cyclic kernel), while for $\sigma=(-1,-1)$ we write $\mathbb{Z}_p\rtimes_6^d E$ (dihedral kernel).

	\subsection{$E=\Dic_{12}$}\label{dici}

     \begin{center}
    \begin{tabular}{|c|c|c|}
\multicolumn{3}{c}{$E=\langle x,y \mid x^3=1, y^4=1, yxy^{-1}=x^2\rangle$}\\[1ex]
\hline
$\Aut(E) $&$\langle g_1,\ g_2\rangle\simeq D_{2\cdot 6}$ &
	 $g_1(x)=x^2,g_1(y)=y^3,g_2(x)=x,g_2(y)=xy^3$\\[1ex] 
$\Hom(E,\Z_p^*)$ & $\mu_{D}=\langle \zeta_D\rangle$, $D=\gcd(4,p-1)$& $\sigma \to  \sigma(y)=\zeta_d^j\rightsquigarrow (d,j\bmod d)$\\
\hline
\end{tabular}
  \end{center} 
	
	The exponent of the dicyclic group is $6$ (the element $xy$ has order $6$)
	and the normal subgroups with cyclic quotient are $\langle x^2\rangle$ and the derived subgroup $\langle x\rangle$. Therefore, $\sigma(x)=1$ for all
	$\sigma\in \Hom(E,\Z_p^*)$ and the image of $y$ determines $\sigma$. The action of $\Aut(E)$ goes as follows:
	\begin{center}
		\begin{tabular}{c|l|l}
			$k$  & Orbit of $\sigma=(d,j\bmod d)$& Stabiliser $\Sigma_{\sigma}$ \\
			\hline
			12  & $(1,1)$   & $\Aut(E)$\\
			6& $(2,1)$   & $\Aut(E)$\\
			3&  $(4,1)\xrightarrow[g_2^3]{g_1}(4,3)$   & $\langle g_2^2, g_1g_2\rangle\simeq D_{2\cdot 3}$\\
		\end{tabular}
	\end{center}

	\section{Skew braces of size $12p$}\label{braces12p}
	
	Once we have done the precomputations for all groups $E$, we bring in the multiplicative group $F$ of the brace $B_{12}$.
	In all the sub-cases we follow the procedure described in Section \ref{method} and the notations in Section \ref{twelve}.

	\subsection{$F=C_{12}$}
	
	According to the table in Section \ref{twelve} and Theorem \ref{npgivesn}, for a brace of size 12 with multiplicative group  $C_{12}$, the additive group can
	be any group except for $A_4$. In all the sub-cases the group $\Hom(F,\Z_p^*)$ is as in \ref{ciclic} so that
	$\tau$ identifies as a pair $(d,j)$ meaning that the given generator of $F$ has image $\zeta_d^j.$

	\subsubsection*{Case $E=C_{12}$}
	
	Let $E=\langle c\rangle$. There is a unique conjugacy class of regular subgroups of $\Hol(E)$ isomorphic to $C_{12}$. In the following table we summarize all the information needed for the algorithmic procedure:

\begin{center}
    \begin{tabular}{|l|l|}
    \hline
      Representative& $F=\langle\  (c,\Id)\ \rangle $ \\
      \hline
      $\pi_2(F)\subseteq \Sigma_{\sigma}$ &  $\forall\sigma$ since $\pi_2(F)=\{Id\}$ \\
      $\Hom(B_{12},\Z_p^*)$ & $\Hom(E,\Z_p^*)$  \\
      $\Phi_g(F)=F$ & $\forall g$ since $F$ is normal in $\Hol(E)$\\
      $\Aut(B_{12})$ & $\Aut(E)$\\
      Orbits in $\Hom(B_{12},\Z_p^*)$ & As in \ref{ciclic}: one for each size $k$\\
      \hline
    \end{tabular}
\end{center}

	Finally we have to compute the orbits of the action  of $\Sigma_{\sigma}$ on $\Hom(F,\Z_p^*)$.
	
	Since $\Phi_{g_i}(c,\Id)=(g_i(c),g_i \Id g_i^{-1})=(c^i,\Id)$,
	two elements $(d,j)$ and $(d,j')$ in $\Hom(F,\Z_p^*)$ are at the same orbit if and only if $j\equiv ij'\bmod d$ for some $i$ with $g_i\in \Sigma_{\sigma}$.
	If we count the number of orbits according to the size $\frac{12}d$ of $\Ker(d,j)$, we obtain
	\begin{center}
		\begin{tabular}{c|l|cccccc}
			$k_{\sigma}$ & $\Sigma_{\sigma}$   & 12 &6 &4&3&2&1  \\
			\hline
			$12,6$ & $\langle g_5,g_7\rangle$ & 1 &1 &1&1&1&1\\
			$4,2$&  $\langle g_7\rangle$& 1&1&2&1&2&2 \\
			$3$ &  $\langle g_5\rangle$& 1 & 1& 1& 2& 1&2  \\
			$1$ &  $\{1\}$& 1 &1 &2&2&2 &4\\
		\end{tabular}
	\end{center}
	For $k_{\sigma}\ne 12$, the number of multiplicative structures is twice the number of orbits.

	\begin{proposition}
		Let $p$ be a prime number, $p\geq 7$. The number of skew braces with additive group $N:=\Z_p \rtimes_k C_{12}$ and multiplicative group
		$G:=\Z_p\rtimes_{k'} C_{12}$ is as shown in the following table, where we need $p\equiv 1 \pmod{\frac {12}k}$ for a kernel of size $k$ to occur.
		
		\begin{center}
			
			\begin{tabular}{|c||c|c|c|c|c|c|}
				\hline
				$N \backslash G$ & $\Z_p\times C_{12}$ & $\Z_p\rtimes_{6} C_{12}$ & $\Z_p\rtimes_{4} C_{12}$ & $\Z_p\rtimes_{3} C_{12}$ & $\Z_p\rtimes_{2} C_{12}$& $\Z_p\rtimes_{1} C_{12}$\\
				\hline
				\hline
				$\Z_p \times C_{12}$ &1&1&1&1 &1&1\\
				\hline
				$\Z_p \rtimes_{6} C_{12}$ &2&2&2 &2 &2 &2\\
				\hline
				$\Z_p \rtimes_{4} C_{12}$ &2&2& 4 &2  &4 &4\\
				\hline
				$\Z_p \rtimes_{3} C_{12}$ &2& 2 & 2&4 &2 &4 \\
				\hline
				$\Z_p \rtimes_{2} C_{12}$ &2&2& 4 &2 &4&4\\
				\hline
				$\Z_p \rtimes_{1} C_{12}$ &2& 2& 4&4 &4 & 8 \\
				\hline
				
			\end{tabular}
		\end{center}
		
	\end{proposition}

	\subsubsection*{Case $E=C_{6}\times C_2$}
    
    Let $E=\langle a\rangle \times \langle b\rangle$.  As before, there is a unique conjugacy class of regular subgroups of $\Hol(E)$ isomorphic to $C_{12}$.

\begin{center}
    \begin{tabular}{|l|l|}
    \hline
      Representative& $F=\langle c=(a,g_1)\rangle$ \\
      \hline
      $\pi_2(F)\subseteq \Sigma_{\sigma}$ &  $\sigma=(1 \bmod d,0\bmod 2)$ for every $d|6$, since $\pi_2(F)=\langle g_1\rangle$ \\
      $\Aut(B_{12})$ & $\langle g_2^3,g_1\rangle=\operatorname{Cent}_{\Aut(E)}(g_1)$.\quad  
      $\Phi_{g_2^3}(c)=c^5$, $\Phi_{g_1}(c)=c^7$ \\
      Orbits in $\Hom(B_{12},\Z_p^*)$ & $4$, one for each size $k$\\
      \hline
    \end{tabular}
\end{center}

	We consider the orbits of $(d,j)\in\Hom(F,\Z_p^*)$. 
	For $k_{\sigma}=12,6$, we have $\langle g_2^3,g_1\rangle {\cap}\Sigma_{\sigma}=\langle g_2^3,g_1\rangle$ and there is a single orbit.  If $k_{\sigma}=4,2$, we have $\langle g_2^3,g_1\rangle {\cap}\Sigma_{\sigma}=
	\langle g_1\rangle$ and the orbits are $\{(d,j),(d,7j\bmod d)\}$.
	We have the following representatives for the orbits:
	\begin{center}
		\begin{tabular}{c|c}
			$k_{\sigma}$ &  $(d,j)$   \\[0.5ex]
			\hline
			$12,6$  & $(d,1)$, $\quad d|12$\\
			$4,2$  & $(1,1)$, $(2,1)$, $(3,1)$, $(3,2)$ (4,1),  (6,1),
			(6,5), (12,1) (12,5) \\
		\end{tabular}
	\end{center}

	\begin{proposition}
		Let $p$ be a prime number, $p\geq 7$. The number of skew braces with additive group $N:=\Z_p \rtimes (C_{6}\times C_{2})$ and multiplicative group $G:=\Z_p\rtimes C_{12}$ is as shown in the following table, where we need $p\equiv 1 \pmod{\frac {12}k}$ for a kernel of size $k$ to occur.
		
		\begin{center}
			
			\begin{tabular}{|c||c|c|c|c|c|c|}
				\hline
				$N \backslash G$ & $\Z_p\times C_{12}$ & $\Z_p\rtimes_{6} C_{12}$ & $\Z_p\rtimes_{4} C_{12}$ & $\Z_p\rtimes_{3} C_{12}$ & $\Z_p\rtimes_{2} C_{12}$& $\Z_p\rtimes_{1} C_{12}$\\
				\hline
				\hline
				$\Z_p\times (C_6\times C_2)$ & 1 & 1&1&1&1&1\\
				\hline
				$\Z_p\rtimes_{6} (C_6\times C_2)$ &2&2&2&2&2&2\\
				\hline
				$\Z_p\rtimes_{4} (C_6\times C_2)$&2 &2&4&2&4&4\\
				\hline
				$\Z_p\rtimes_{2} (C_6\times C_2)$&2 &2&4&2&4&4\\
				\hline
			\end{tabular}
		\end{center}
		
	\end{proposition}

	\subsubsection*{Case $E=D_{2\cdot 6}$}
	
	Let $E=\langle r,s\rangle$.
	There are two conjugacy classes of regular subgroups of $\Hol(D_{2\cdot 6})$ isomorphic to $C_{12}$.

\begin{center}
    \begin{tabular}{|l|l|l|}
    \hline
      Representative& $F_1=\langle\  c_1=(s, g_2)\ \rangle $ & $F_2=\langle\  c_2=(rs,g_2g_1 )\ \rangle$ \\
      \hline
      $\pi_2(F)$ & $\langle g_2\rangle \simeq C_6$ &$ \langle g_2g_1\rangle \simeq C_2$\\
      $\pi_2(F)\subseteq \Sigma_{\sigma}$ &{$\sigma=(1,\pm 1)$} &{$\sigma=(1,\pm 1)$}  \\
      $\Phi_g(F)=F$ &{$g\in\langle g_2^3,g_2^2g_1\rangle$}&{$g\in\langle g_2^3,g_2^2g_1\rangle$}\\
      Orbits in $\Hom(B_{12},\Z_p^*)$ & $2$ & $2$ \\
      \hline
    \end{tabular}
\end{center}    

Since $\Sigma_{\sigma}=\Aut(E)$, the action on $\Hom(F_i,\Z_p^*)$ is given by $\Aut(B_{12})=\langle g_2^3,g_2^2g_1\rangle$.
We have
	$$\begin{array}{lclcl}
    \Phi_{g_2^3}(c_1)=c_1^7, && \Phi_{g_2^2g_1}(c_1)=c_1^5,&&\Phi_{g_2^5g_1}(c_1)=c_1^{11}\\
	\Phi_{g_2^3}(c_2)=c_2^7,&& \Phi_{g_2^2g_1}(c_2)=c_2^{11},&&
	\Phi_{g_2^5g_1}(c_2)=c_2^{5}.
    \end{array}$$
	Therefore, in $\Hom(F_i,\Z_p^*)$ we have just one orbit for each $d$.
	
	\begin{proposition}
		Let $p$ be a prime number, $p\geq 7$.
		The number of skew braces with additive group $N:=\Z_p \rtimes D_{2.6}$ and multiplicative group $G:=\Z_p\rtimes C_{12}$ is as shown in the following table, where we need
		$p\equiv 1 \pmod{\frac {12}k}$ for a kernel of size $k$ to occur.

		\begin{center}
			\begin{tabular}{|c||c|c|c|c|c|c|}
				\hline
				$N \backslash G$ & $\Z_p\times C_{12}$ & $\Z_p\rtimes_{6} C_{12}$ & $\Z_p\rtimes_{4} C_{12}$
				& $\Z_p\rtimes_{3} C_{12}$ & $\Z_p\rtimes_{2} C_{12}$& $\Z_p\rtimes_{1} C_{12}$\\
				\hline
				\hline
				$\Z_p \times D_{2\cdot 6}$ &2 & 2 & 2 &2 &2 & 2 \\
				\hline
				$\Z_p \rtimes_6^c D_{2\cdot 6}$ &4 & 4 & 4& 4&4  & 4 \\
				\hline
				
			\end{tabular}
		\end{center}
	\end{proposition}

	\subsubsection*{Case $E=\Dic_{12}$}
	
	Let $E=\langle \  x,y\ \rangle$. Again, 
	there are two conjugacy classes of cyclic regular subgroups of $\Hol(\Dic_{12})$.

\begin{center}
    \begin{tabular}{|l|l|l|}
    \hline
      Representative& $F_1=\langle\  c_1=(y, g_2g_1) \ \rangle$ & $F_2=\langle\  c_2=(y,g_2^2) \ \rangle$ \\
      \hline
      $\pi_2(F)$ & $\langle g_2g_1\rangle\subseteq \langle g_2^2, g_1g_2\rangle$ &$\langle g_2^2\rangle\subseteq \langle g_2^2, g_1g_2\rangle$\\
      $\pi_2(F)\subseteq \Sigma_{\sigma}$ & $\forall \sigma$ &
      $\forall \sigma$ \\
      $\Phi_g(F)=F$ &$g\in\langle g_2^3,g_2g_1\rangle$ &$g\in\langle g_2^3,g_2g_1\rangle$\\
      Orbits in $\Hom(B_{12},\Z_p^*)$ & One for each order & One for each order \\
      \hline
    \end{tabular}
\end{center}

For $\sigma=(1,1)$ or $(2,1)$ in $\Hom(E,\Z_p^*)$, we have
$\Aut(B_{12})\cap \Sigma_{\sigma}=\Aut(B_{12})$, which acts on
$\Hom(F,\Z_p^*)$ giving a single orbit for each order $d$,
since
	$$
	\begin{array}{lcl}
	\Phi_{g_2^3}(c_1)=(y^3,g_2g_1)=c_1^7,&& \Phi_{g_2^3}(c_2)=(y^3,g_2^2)=c_2^7,\\
	\Phi_{g_2g_1}(c_1)=(xy,g_2g_1)=c_1^5, &&\Phi_{g_2g_1}(c_2)=(xy, g_2^4)=c_2^5.
\end{array}
$$
For $\sigma=(4,1)$, we have
$\Aut(B_{12})\cap \Sigma_{\sigma}=\langle g_2g_1\rangle$ 
and $(d,j)$ and $(d,j')$ are in the same orbit if and only if $j\equiv 5j'\bmod d$, so that we have orbits represented by
$
(1,1),(2,1),(3,1),(4,1),(4,3),(6,1),(12,1),(12,7).
$

\begin{proposition}
	Let $p$ be a prime number, $p\geq 7$.
	The number of skew braces with additive group $N:=\Z_p \rtimes \Dic_{12}$ and multiplicative group $G:=\Z_p\rtimes C_{12}$ is as shown in the following table, where we need $p\equiv 1 \pmod{\frac {12}k}$ for a kernel of size $k$ to occur.
	
	\begin{center}
		\begin{tabular}{|c||c|c|c|c|c|c|}
			\hline
			$N \backslash G$ & $\Z_p\times C_{12}$ & $\Z_p\rtimes_{6} C_{12}$ & $\Z_p\rtimes_{4} C_{12}$ & $\Z_p\rtimes_{3} C_{12}$ & $\Z_p\rtimes_{2} C_{12}$& $\Z_p\rtimes_{1} C_{12}$\\
			\hline
			\hline
			$\Z_p \times \Dic_{12}$ &2 & 2 & 2 &2 &2 & 2 \\
			\hline
			$\Z_p \rtimes_{6} \Dic_{12}$ &4& 4  & 4& 4& 4 & 4 \\
			\hline
			$\Z_p \rtimes_{3} \Dic_{12}$ &4 & 4 & 4 & 8 & 4 & 8 \\
			\hline
		\end{tabular}
	\end{center}
\end{proposition}

\subsection{$F=C_{6}\times C_2$}

According to the table in Section \ref{twelve} and Theorem \ref{npgivesn}, for a brace of size 12 with multiplicative group  $C_{12}$, the additive group can
be any group.  In all the sub-cases the group $\Hom(F,\Z_p^*)$ is as in \ref{C62} so that
$\tau$ identifies as a pair $(j\bmod d ,i\bmod 2)$ meaning
$\tau(a)=\zeta_d^j, \ \tau(b)=(-1)^i$ for $F=\langle a\rangle \times \langle b\rangle $.

\subsubsection*{Case $E=C_{12}$}

Let $E=\langle c\rangle$. There is a unique conjugacy class of regular subgroups of $\Hol(E)$ isomorphic to $C_6\times C_2$.  

\begin{center}
    \begin{tabular}{|l|l|}
    \hline
      Representative& $F=\langle  a=(c,g_7) ,\ b=(c^6,\Id)\rangle$ \\
      \hline
      $\pi_2(F)\subseteq \Sigma_{\sigma}$ &  $\sigma=(d,j)$ for $4\nmid d$, since $\pi_2(F)=\langle g_7\rangle$ \\
      $\Aut(B_{12})$ & $\Aut(E)$ since $F$ is normal in $\Hol(E)$\\
      Orbits in $\Hom(B_{12},\Z_p^*)$ & $4$, one for each size $k$\\
      \hline
    \end{tabular}
\end{center}    
    
Regarding the action on  $\Hom(F,\Z_p^*)$, since
$$\begin{array}{lcl}\Phi_{g_5}(a)=(c^5,g_7)=a^5, & & \Phi_{g_5}(b)=(c^6,\Id)=b, \\[0.5ex]
\Phi_{g_7}(a)=(c^7,g_7)=ab, && \Phi_{g_7}(b)=(c^6,\Id)=b, \\[0.5ex]
\Phi_{g_{11}}(a)=(c^{11},g_7)=a^5b, && \Phi_{g_{11}}(b)=(c^6,\Id)=b. \end{array}
$$
we have the following orbits, acording to the order of $\tau=(j\bmod d ,i\bmod 2)$:
$$
\begin{array}{lccccc}
d=1 &&& (1,0)\ \bullet \circlearrowleft &&\\[2ex]
d=2 && & (1,0)\ \bullet \circlearrowleft  &&(0,1)\ \bullet \xleftrightarrow[11]{7}\bullet (1,1)\\[3ex]
d=3 &&&  (1,0)\bullet \xleftrightarrow[11]{5}\bullet (2,0)&&\\[3ex]
d=6 &&& (1,0)\bullet  &&  (1,1)\bullet \xleftrightarrow{7} \bullet (4,1)\\
&& &\ \qquad {_{5}}\Big\updownarrow {_{11}}&&{_{5}}\Big\updownarrow\quad \quad \Big\updownarrow {_{5}}\\
&&&  (5,0)\bullet && (5,1)\bullet \xleftrightarrow{7} \bullet (2,1)\\
\end{array}
$$
In the last two cases, when $\Sigma_{\sigma}=\langle g_7\rangle$ some points get disconnected and we have 2 and 4 orbits, respectively.

\begin{proposition}
	Let $p$ be a prime number, $p\geq 7$. The number of skew braces with additive group $N=\Z_p \rtimes C_{12}$ and multiplicative group $G=\Z_p\rtimes (C_6\times C_2)$ is as shown in the following table, where  we need $p\equiv 1 \pmod{\frac {12}k}$ for a kernel of size $k$ to occur.
	
	\begin{center}
		\begin{tabular}{|c||c|c|c|c|}
			\hline
			$N \backslash G$ & $\Z_p\times (C_6\times C_2)$ & $\Z_p\rtimes_{6} (C_6\times C_2)$ & $\Z_p\rtimes_{4} (C_6\times C_2)$ & $\Z_p\rtimes_{2} (C_6\times C_2)$ \\
			\hline
			\hline
			$\Z_p \times C_{12}$ &1&2&1&2\\
			\hline
			$\Z_p \rtimes_{6} C_{12}$ &2&4&2&4\\
			\hline
			$\Z_p \rtimes_{2} C_{12}$ &2&4&4&8\\
			\hline
			$\Z_p \rtimes_{4} C_{12}$ &2&4&4&8\\
			
			\hline
		\end{tabular}
	\end{center}
\end{proposition}

\subsubsection*{Case $E=C_{6}\times C_2$}

Let $E=\langle a\rangle \times \langle b\rangle$. There is  one conjugacy class of regular subgroups of $\Hol(E))$ isomorphic to $C_6\times C_2$.

\begin{center}
    \begin{tabular}{|l|l|}
    \hline
      Representative& $F=\langle (a,\Id), (b,\Id)\rangle$ \\
      \hline
      $\pi_2(F)\subseteq \Sigma_{\sigma}$ &  $\forall\sigma$ since $\pi_2(F)=\{Id\}$ \\
       Orbits in $\Hom(B_{12},\Z_p^*)$ &  One for each order\\
      \hline
    \end{tabular}
\end{center}   

Regarding the action on $\Hom(F,\Z_p^*)$, for $\sigma=1$ we have one orbit for each order and for $\tau=1$ the orbit has a single point. According to the stabilisers and orbits computed in \ref{C62}, in the remaining cases we have
\begin{center}
	\begin{tabular}{r|c l}
		$\Sigma_{(1,0)}$    & $d$ & Orbits \\
		\hline
		$\langle g_2^3, g_1\rangle$  & 2 & $(1,0)\ \bullet \circlearrowleft\quad (0,1)\ \bullet \xleftrightarrow{g_1}\bullet\ (1,1)$\\ [1ex]
		& 3 & $(1,0)\ \bullet \xleftrightarrow{g_2^3}\bullet\ (2,0)$\\ [1ex]
		&6 & $(1,0)\ \bullet \xleftrightarrow{g_2^3}\bullet\  (5,0)$\\ [1ex]
		& & $(2,1)\ \bullet \xleftrightarrow{g_2^3}\bullet\   (4,1)\ \bullet \xleftrightarrow{g_1}\bullet\   (1,1)\ \bullet \xleftrightarrow{g_2^3}\bullet  (5,1)$\\ [1ex]
		\hline
		$\langle g_2^2, g_1\rangle$  & 2 &  $(1,0)\ \bullet\xleftrightarrow{g_2^2} \bullet \ (1,1)\ \bullet \xleftrightarrow{g_1}\bullet\ (0,1)$\\ [1ex]
		& 3 & $(1,0)\ \bullet \circlearrowleft \qquad \bullet\ (2,0) \circlearrowleft$\\ [1ex]
		&6 & $(1,0)\ \bullet \xleftrightarrow{g_2^2}\bullet\  (1,1)\   \bullet \xleftrightarrow{g_2^2}
		\bullet\  (4,1)$\\ [1ex]
		& & $(5,1)\ \bullet \xleftrightarrow{g_2^2}\bullet\  (2,1)\   \bullet \xleftrightarrow{g_2^2}
		\bullet\  (5,0)$\\
		\end{tabular}
\end{center}
        \begin{center}
	\begin{tabular}{r|c l}
		$\Sigma_{(1,0)}$    & $d$ & Orbits \\
		\hline
		$\langle g_1\rangle$  &2 & $(1,0)\ \bullet \circlearrowleft\quad (0,1)\ \bullet \xleftrightarrow{g_1}\bullet\ (1,1)$\\ [2ex]
		& 3 &$(1,0)\ \bullet \circlearrowleft \qquad \bullet\ (2,0) \circlearrowleft$\\ [2ex]
		&6 & $(1,0)\ \bullet \circlearrowleft \qquad \bullet\ (5,0) \circlearrowleft$\\ [2ex]
		& & $(1,1)\ \bullet \xleftrightarrow{g_1}\bullet\ (4,1)$
		\qquad $(2,1)\ \bullet \xleftrightarrow{g_1}\bullet\ (5,1)$\\ [2ex]
\end{tabular}
\end{center}

\begin{proposition}
	Let $p$ be a prime number, $p\geq 7$. The number of skew braces with additive group $N=\Z_p \rtimes (C_{6}\times C_2)$ and multiplicative group $G=\Z_p\rtimes (C_6\times C_2)$ is as shown in the following table, where we need $p\equiv 1 \pmod{\frac {12}k}$ for a kernel of size $k$ to occur.
	
	\begin{center}
		\begin{tabular}{|c|c|c|c|c|}
			\hline
			$N \backslash G$ & $\Z_p\times (C_6\times C_2)$ & $\Z_p\rtimes_{6} (C_6\times C_2)$ & $\Z_p\rtimes_{4} (C_6\times C_2)$ & $\Z_p\rtimes_{2} (C_6\times C_2)$ \\
			\hline
			\hline
			$\Z_p\times (C_6\times C_2)$ &1&1&1&1\\
			\hline
			$\Z_p\rtimes_{6} (C_6\times C_2)$ &2&4&2&4\\
			\hline
			$\Z_p\rtimes_{4} (C_6\times C_2)$ &2&2&4&4\\
			\hline
			$\Z_p\rtimes_{2} (C_6\times C_2)$ &2&4&4&8\\
			\hline
			
		\end{tabular}
	\end{center}
\end{proposition}

\subsubsection*{Case $E=A_4$}

There are two conjugacy classes of regular subgroups of $\Hol(A_4)$ isomorphic to $C_6\times C_2$.  

\begin{center}
    \begin{tabular}{|l|l|l|}
    \hline
      Representative& $F_1=\langle a_1,b_1 \rangle$ & $F_2=\langle\  a_2,b_2 \ \rangle$ \\
      \hline
      $\pi_2(F)$ & $\langle \phi_{(1,2,4)}\rangle $ &$\langle \phi_{(1,2)(3,4)},\phi_{(1,3)(2,4)}\rangle$\\
      $\pi_2(F)\subseteq \Sigma_{\sigma}$ & $\forall \sigma$ &
      $\forall \sigma$ \\
      $\Aut(B_{12})$ &$\langle \phi_{(1,2)}, \phi_{(1,2,4)}\rangle$ &$\langle \phi_{(1,2)}, \phi_{(1,2,3)}\rangle$\\
      Orbits in $\Hom(B_{12},\Z_p^*)$ & One for each order & One for each order \\
      \hline
    \end{tabular}
\end{center}    
where 
$$\begin{array}{lcl}
a_1=((1,2,3),\phi_{(1,2,4)}),&& b_1=((1,2)(3,4),\Id) \\
a_2=((1,3,4),\phi_{(1,2)(3,4)}),&& b_2=((1,3)(2,4),\phi_{(1,3)(2,4)}).
\end{array}$$

We have 
$$
 \{s\in S_4 : \Phi_{\phi_s} (F_i)=F_i\}\cap A_4=
\left\{
\begin{array}{l}
\langle (1,2,4)\rangle\  \mbox{ for }  i=1\\
\langle  (1,2,3)\rangle\ \mbox{ for }  i=2.
\end{array}\right.
$$
and in both cases 
$\phi_c(a)=ab$, $\phi_c(b)=a^3$, for $c$ the corresponding cycle of length 3. The action on $\Hom(F,\Z_p^*)$ is given by
$$
(j \bmod d, i\bmod 2)\Phi_{\phi_{c}}(a)=(-1)^i\zeta_d^{j},
\qquad (j \bmod d, i\bmod 2)\Phi_{\phi_c}(b)=\zeta_d^{3j}
$$
and we obtain the following orbits:
\begin{center}
	\begin{tabular}{lll l }
		$d=2$ &  & $(1,0)\ \bullet \xleftrightarrow{} \bullet\  (1,1)\ \bullet \xleftrightarrow{} \bullet\   (0,1)$ \\[1ex]
		$d=3$& &  $(1,0)\ \bullet \circlearrowleft \qquad \bullet\ (2,0) \circlearrowleft$ \\[1ex]
		$d=6$ & & $(1,0)\ \bullet \xleftrightarrow{}\bullet\   (1,1)\ \bullet \xleftrightarrow{}\bullet\   (4,1)$  \\[1ex]
		& &  $(5,0)\ \bullet \xleftrightarrow{}\bullet\   (5,1)\ \bullet \xleftrightarrow{}\bullet\   (2,1)$
	\end{tabular}
\end{center}

\begin{proposition}
	Let $p$ be a prime number, $p\geq 7$. The number of skew braces with additive group $N:=\Z_p \rtimes A_4$ and multiplicative group $G:=\Z_p\rtimes (C_6\times C_2)$ is as shown in the following table,  where we need $p\equiv 1 \pmod{\frac {12}k}$ for a kernel of size $k$ to occur.
	
	\begin{center}
		\begin{tabular}{|c||c|c|c|c|}
			\hline
			$N \backslash G$ & $\Z_p\times (C_6\times C_2)$ & $\Z_p\rtimes_{6} (C_6\times C_2)$ & $\Z_p\rtimes_{4} (C_6\times C_2)$ & $\Z_p\rtimes_{2} (C_6\times C_2)$ \\
			\hline
			\hline
			$\Z_p \times A_4$ &2&2&2&2\\
			\hline
			$\Z_p \rtimes A_4$ &4&4&8&8\\
			\hline
			
		\end{tabular}
	\end{center}
\end{proposition}

\subsubsection*{Case $E=D_{2\cdot 6}$}

Let $E=\langle r,s\rangle$. There are two conjugacy classes of regular subgroups of $\Hol(E)$ isomorphic to $C_6\times C_2$.

\begin{center}
    \begin{tabular}{|l|l|l|}
    \hline
      Representative& $F_1=\langle a_1=(r,\Id),\ b_1=(s,g_1) \rangle$ & $F_2=\langle\  a_2=(r,g_2^4),\ b_2=(s,\Id) \ \rangle$ \\

      \hline
      $\pi_2(F)$ & $\langle g_1 \rangle\simeq C_2 $ &$\langle g_2^4\rangle\simeq C_3$\\
      $\pi_2(F)\subseteq \Sigma_{\sigma}$ & $\forall \sigma$ &
      $\forall \sigma$ \\
      $\Aut(B_{12})$ &$\langle g_2^3,g_1\rangle$ &$\langle g_2^3,g_1\rangle$ \\
      Orbits in $\Hom(B_{12},\Z_p^*)$ & One for each order & One for each order \\
      \hline
    \end{tabular}
\end{center}    

Since 
$$\begin{array}{lll}
\Phi_{g_1}(a_1)=(r^5,\Id)=a_1^5,& &  \Phi_{g_1}(b_1)=(s,g_1)=b_1, \\
\Phi_{g_2^3}(a_1)=(r,\Id)=a_1,& &  \Phi_{g_2^3}(b_1)=(r^3,g_1)=a_1^3 b_1,\\
\Phi_{g_1}(a_2)=(r^5,g_2^2)=a_2^5,& &  \Phi_{g_1}(b_2)=(s,\Id)=b_2, \\
\Phi_{g_2^3}(a_2)=(r,g_2^4)=a_2,& &  \Phi_{g_2^3}(b_2)=(r^3s,\Id)=a_2^3 b_2,\\
\end{array}
$$
there is no difference in computation of orbits for $F_1$ or $F_2$. We obtain
\begin{center}
	\begin{tabular}{lll l }
		$d=2$ &  &  $(0,1)\ \bullet \circlearrowleft\qquad (1,0)\ \bullet \xleftrightarrow{g_2^3}\bullet\  (1,1)$ \\[2ex]
		$d=3$& &  $(1,0)\ \bullet \xleftrightarrow{g_1}\bullet\  (2,0)$  \\[2ex]
		$d=6$ & & $ (2,1)\ \bullet \xleftrightarrow{g_1}\bullet\  (4,1)$  \\[2ex]
		& &  $(1,0)\ \bullet \xleftrightarrow{g_1}\bullet\   (5,0)$ \\
		&&$\quad\ {_{g_2^3}}\Big\updownarrow\qquad\  \Big\updownarrow {_{g_2^3}}$ & \\
		& &  $(1,1)\ \bullet \xleftrightarrow{g_1}\bullet\   (5,1)$ \\
	\end{tabular}
\end{center}
For $\sigma$ of order $6$ with dihedral kernel $g_2^3\not\in\Sigma_{\sigma}$ and in first and last case
orbits split.

\medskip

\begin{proposition}
	Let $p$ be a prime number, $p\geq 7$.
	The number of skew braces with additive group $N:=\Z_p \rtimes D_{2.6}$ and multiplicative group $G:=\Z_p\rtimes (C_6\times C_2)$ is as shown in the following table,
	where we need $p\equiv 1 \pmod{\frac {12}k}$ for a kernel of size $k$ to occur.
	
	\begin{center}
		\begin{tabular}{|c||c|c|c|c|}
			\hline
			$N \backslash G$ & $\Z_p\times (C_6\times C_2)$ & $\Z_p\rtimes_{6} (C_6\times C_2)$ & $\Z_p\rtimes_{4} (C_6\times C_2)$ & $\Z_p\rtimes_{2} (C_6\times C_2)$ \\
			\hline
			\hline
			$\Z_p \times D_{2\cdot 6}$ &2&4&2&4\\
			\hline
			$\Z_p \rtimes_{6}^c D_{2\cdot 6}$ &4&8&4&8\\
			\hline
			$\Z_p \rtimes_{6}^d D_{2\cdot 6}$ &4&12&4&12\\
			\hline
			
		\end{tabular}
	\end{center}
\end{proposition}

\subsubsection*{Case $E=\Dic_{12}$}

Let $E=\langle x,y\rangle$. There are two conjugation classes of regular subgroups of $\Hol(E)$ isomorphic to $C_6\times C_2$.

\begin{center}
    \begin{tabular}{|l|l|l|}
    \hline
      Representative& $F_1=\langle a_1=(xy,g_1),\ \,b_1=(y^2,\Id)\rangle$ & $F_2=\langle a_2=(xy^2,\Id),\ b_2=(y,g_1)\rangle$ \\    
      \hline
      $\pi_2(F)$ & $\langle g_1 \rangle $ &$\langle g_1\rangle$\\
      $\pi_2(F)\subseteq \Sigma_{\sigma}$ & $\sigma=(1,1),(2,1)$  &
      $\sigma=(1,1),(2,1)$\\
      $\Aut(B_{12})$ &$\langle g_1,g_2^3\rangle$ &$\langle g_1,g_2^3\rangle$\\
      Orbits in $\Hom(B_{12},\Z_p^*)$ & One for each order & One for each order \\
      \hline
    \end{tabular}
\end{center}    

 We have
$$\begin{array}{lll}
\Phi_{g_1}(a_1)=(x^2y^3, g_1)=a_1^5b_1, & &
\Phi_{g_1}(b_1)=(y^2,\Id)=b_1\\
\Phi_{g_2^3}(a_1)=(xy^3,g_1)=a_1b_1  &&  \Phi_{g_2^3}(b_1)=(y^2,\Id)=b_1\\
\Phi_{g_1}(a_2)= (x^2y^2,\Id)=a_2^5 & &\Phi_{g_1}(b_2)=(y^3,g_1)=a_2^3 b_2\\
\Phi_{g_2^3}(a_2)=(xy^2,\Id)=a_2 &&    \Phi_{g_2^3}(b_2)=(y^3,g_1)=a_2^3b_2.\\
\end{array}
$$
and the following orbits
\begin{center}
	\begin{tabular}{lll l r}
		$d=2$ &  &  $(1,0)\ \bullet \circlearrowleft\qquad (0,1)\ \bullet \xleftrightarrow[g_1]{g_2^3}\bullet\  (1,1)$ && $F_1$\\[2ex]
		&  &  $(0,1)\ \bullet \circlearrowleft\qquad (1,0)\ \bullet \xleftrightarrow[g_1]{g_2^3}\bullet\  (2,0)$ & &$F_2$\\[2ex]
		$d=3$& &  $(1,0)\ \bullet \xleftrightarrow{g_1}\bullet\  (2,0)$ & &$F_1,F_2$\\[2ex]
		$d=6$ & & $ (1,0)\ \bullet \xleftrightarrow{g_1}\bullet\  (5,0)$ &
		$(1,1)\ \bullet \xleftrightarrow{g_1}\bullet\   (2,1)$ &$F_1$  \\[0.1ex]
		&&&$\quad\ {_{g_2^3}}\Big\updownarrow\qquad\  \Big\updownarrow {_{g_2^3}}$  \\
		& &&  $(4,1)\ \bullet \xleftrightarrow{g_1}\bullet\   (5,1)$ \\[2ex]
		& & $ (2,1)\ \bullet \xleftrightarrow{g_1}\bullet\  (4,1)$ &
		$(1,0)\ \bullet \xleftrightarrow{g_1}\bullet\   (5,1)$ &$F_2$  \\[0.1ex]
		&&&$\quad\ {_{g_2^3}}\Big\updownarrow\qquad\  \Big\updownarrow {_{g_2^3}}$  \\
		& &&  $(1,1)\ \bullet \xleftrightarrow{g_1}\bullet\   (5,0)$ \\
	\end{tabular}
\end{center}
Therefore, both $F_1$ and $F_2$ give rise to the same number of multiplicative structures.

\begin{proposition}
	Let $p$ be a prime number, $p\geq 7$.
	The number of skew braces with additive group $N=\Z_p \rtimes \Dic_{12}$ and multiplicative group $G=\Z_p\rtimes (C_6\times C_2)$ is as shown in the following table,
	where we need $p\equiv 1 \pmod{\frac {12}k}$ for a kernel of size $k$ to occur.
	
	\begin{center}
		\begin{tabular}{|c||c|c|c|c|}
			\hline
			$N \backslash G$ & $\Z_p\times (C_6\times C_2)$ & $\Z_p\rtimes_{6} (C_6\times C_2)$ & $\Z_p\rtimes_{4} (C_6\times C_2)$ & $\Z_p\rtimes_{2} (C_6\times C_2)$ \\
			\hline
			\hline
			$\Z_p\times \Dic_{12}$ &2&4&2&4\\
			\hline
			$\Z_p\rtimes_{6} \Dic_{12}$ &4&8&4&8\\
			\hline
		\end{tabular}
	\end{center}
\end{proposition}

\subsection{$F=A_4$}
According to the table in Section \ref{twelve} and Theorem \ref{npgivesn}, for a brace of size 12 with multiplicative group  $C_{12}$, the additive group can be, the additive group is either  $C_6\times C_2$ or $A_4$. In both sub-cases the group
$\Hom(F,\Z_p^*)$ is as in \ref{A4} so that its elements are $\tau_j$,  with $j=0,1,2$, defined by $\tau_j(1,2,3)=\zeta_3^j$ and $\tau_j$ trivial on elements of order 2.

\subsubsection*{Case $E=C_6 \times C_2$}

Let $E=\langle a\rangle \times \langle b \rangle$. There is 
one conjugacy class of regular subgroups of $\Hol(E)$ isomorphic to $A_4$.

\begin{center}
    \begin{tabular}{|l|l|}
    \hline
      Representative& $F=\langle (a,g_2^2),(b, \Id)\rangle$ \\
      \hline
      $\pi_2(F)$& $ \langle g_2^2\rangle $\\
      $\pi_2(F)\subseteq \Sigma_{\sigma}$ &  $\sigma$ of order $1$ or $3$ \\
      Orbits in $\Hom(B_{12},\Z_p^*)$ & One for each order\\
      \hline
    \end{tabular}
\end{center}    
Since $\Phi_{g_1g_2}(a,g_2^2)=(a^2b,g_2^4)=(a,g_2^2)^2$ and 
$\Phi_{g_1g_2}(b,\Id)=(a^3,\Id)\in F
$
we see that $g_1g_2\in\Aut(B_{12})$. For $\sigma=(1,0)$ of order $3$, we have $(1,0)g_1g_2=(1,0)g_2=(2,0)$.

On the other hand,
$$
\Phi_{g_2^2}(a)=(a^2b, g_2^2)\not\in F,\quad \Phi_{g_1}(a)=(ab,g_2^4)\not\in F
$$
so that none of the elements of $\Sigma_{\sigma}=\langle g_2^2, g_1\rangle$ are brace automorphisms.
The action on $\Hom(F,\Z_p^*)$ is just given by the identity and orbits consist on single elements. For a non trivial $\sigma$, $\tau_1$ and $\tau_2$ give different multiplicative structures.

\begin{proposition}
	Let $p$ be a prime number, $p\geq 7$.
	The number of braces with additive group $N:=\Z_p \rtimes (C_6\times C_2)$ and multiplicative group $G:=\Z_p\rtimes A_4$ is as shown in the following table, where we need $p\equiv 1 \pmod{\frac {12}k}$ for a kernel of size $k$ to occur.
	
	\begin{center}
		\begin{tabular}{|c||c|c|}
			\hline
			$N \backslash G$ & $\Z_p\times A_4$ & $\Z_p\rtimes_{4} A_4 $ \\
			\hline
			\hline
			$\Z_p\times (C_6\times C_2)$ & 1&1 \\
			\hline
		\end{tabular}
	\end{center}
\end{proposition}

\subsubsection*{Case $E=A_4$}

There are four  conjugacy classes of regular subgroups of $\Hol(A_4)$ isomorphic to $A_4$. If we write, $\rho=(1,2,3)$ and $\mu=(1,2)(3,4)$, we have:
\begin{center}
    \begin{tabular}{|l|l|l|l|}
    
    \hline
       $F_1=\langle (\rho,\Id),(\mu,\Id)\rangle$ &
     $ F_2=\langle(\rho,\phi_{\rho}),(\mu,\Id)\rangle$&
      $F_3=\langle(\rho,\phi_{\rho^{-1}}),(\mu,\phi_{\mu})\rangle$&
      $F_4=\langle(\rho,\phi_{\rho^{-1}}),(\mu,\phi_{\mu}) \rangle$\\
      \hline
      $\pi_2(F_1)=\{\Id\}$&
      $\pi_2(F_2)=\langle \phi_{\rho}\rangle\simeq C_3$ &
      $\pi_2(F_3)\simeq A_4$&
      $\pi_2(F_4)\simeq A_4$\\
      $\pi_2(F_1)\subseteq \Sigma_{\sigma}$ $\forall \sigma$ & 
      $\pi_2(F_2)\subseteq \Sigma_{\sigma}$ $\forall \sigma$ &
      $\pi_2(F_3)\subseteq \Sigma_{\sigma}$ $\forall \sigma$ &
      $\pi_2(F_4)\subseteq \Sigma_{\sigma}$ $\forall \sigma$
       \\
      \hline
    \end{tabular}
\end{center}    

Since
$$
\Phi_{\phi_{(1,2)}}(\rho, \Id)=(\rho,\Id)^2,\quad 
\Phi_{\phi_{(1,2)}}(\rho, \phi_{\rho})=(\rho, \phi_{\rho})^2,\quad
\Phi_{\phi_{(1,2)}}(\rho, \phi_{\rho^{-1}})=(\rho, \phi_{\rho^{-1}})^2,
$$
and $(\mu, \Id)$ and $(\mu,\phi_{\mu})$ are invariant under $\Phi_{\phi_{(1,2)}}$,
we have $\Phi_{\phi_{(1,2)}}(F_i)=F_i$ for all $i$ and elements of order 3 of $\Hom(E,\Z_p)^*$ are in the same orbit. As for morphisms $\tau\in \Hom(F,\Z_p^*)$, since $\Sigma_{\sigma}$ stabilizes them, the action
gives orbits of a single element, so that we have two orbits in order 3.

\begin{proposition}
	Let $p$ be a prime number, $p\geq 7$.
	The number of skew braces with additive group $N=\Z_p \rtimes A_4$ and multiplicative group $G=\Z_p\rtimes A_4$ is as shown in the following table, where we need $p\equiv 1 \pmod{3}$ for the group
	$\Z_p\rtimes_{4} A_4$ to occur.
	
	\begin{center}
		\begin{tabular}{|c||c|c|}
			\hline
			$N \backslash G$ & $\Z_p\times A_4$ & $\Z_p\rtimes_{4} A_4 $ \\
			\hline
			\hline
			$\Z_p\times A_4$ & 4&4 \\ \hline $\Z_p\rtimes_{4} A_4 $ &8&16\\ \hline
		\end{tabular}
	\end{center}
\end{proposition}

\subsection{$F=D_{2\cdot 6}$}

  From the table in Section \ref{twelve} we know that for a brace of size $12$ with multiplicative group $D_{2\cdot6}$, the additive group can be  any  groups except $A_4$. Moreover, in all the sub-cases the group $\Hom(F,\Z_p^*)$ is as in \ref{D26}, so $\tau$ identifies with any of the pairs $(\pm1,\pm1)$. Let us call $\tau_1=(1,-1)$, $\tau_2=(-1,-1)$ and $\tau_3=(-1,1)$, so that $\tau_1$ has kernel isomorphic to $C_6$ and $\tau_2$, $\tau_3$ have kernel isomorphic to $D_{2\cdot3}$.

\subsubsection*{Case $E=C_{12}$}

Let $E=\langle c\rangle$. There are two conjugacy classes of regular subgroups of $\Hol(E)$ isomorphic to $D_{2\cdot6}$. 

\begin{center}
    \begin{tabular}{|l|l|l|}
    \hline
      Representative& $F_1=\langle  r_1=(c^2,\mathrm{Id}),s_1=(c,g_{11}) \rangle$ & $F_2=\langle\  r_2=(c,g_7),s_2=(c^3,g_{11}) \ \rangle$ \\           \hline
      $\pi_2(F)$ & $\langle g_{11} \rangle $ &$\Aut(E)$\\
      $\pi_2(F)\subseteq \Sigma_{\sigma}$ &  $\sigma=(1,1),(2,1)$ &
      $\sigma=(1,1),(2,1)$ \\
      $\Aut(B_{12})$ &$\Aut(E)$ & $\langle g_5\rangle$ \\
      Orbits in $\Hom(B_{12},\Z_p^*)$ & One for each order & One for each order \\
      \hline
    \end{tabular}
\end{center}    

Since 
$$\Phi_{g_7}(r_1)=(c^2,\Id)=r_1,\   \Phi_{g_7}(s_1)=(c^7,g_{11})=r_1^3s_1,$$
for $F_1$ we have $\tau_2\Phi_{g_7}=\tau_3$. Therefore the action of 
$\Aut(E)=\Aut(B_{12})\cap \Sigma_{\sigma}$ on $\Hom(F_1,\Z_p^
*)$ gives two orbits of order $2$, one with cyclic kernel and  one with dihedral kernel. If we now take $F_2$, we need to consider the action of $\langle g_5\rangle$ on $\Hom(F,\Z_p^*)$. Since
$$\Phi_{g_5}(r_2)=r_2^5,\   \Phi_{g_5}(s_2)=s_2,$$
we have $\tau_2\Phi_{g_5}=\tau_2$ and in order 2 we have one orbit with cyclic kernel and two orbits with dihedral kernel.

\begin{proposition}
	Let $p$ be a prime number, $p\geq 7$.
	The number of braces with additive group $N:=\Z_p \rtimes C_{12}$ and multiplicative group $G:=\Z_p\rtimes D_{2\cdot 6}$ is as shown in the following table.
	
	\begin{center}
		\begin{tabular}{|c||c|c|c|}
			\hline
			$N \backslash G$ & $\Z_p\times D_{2\cdot 6}$ & $\Z_p\rtimes_6^c D_{2\cdot 6} $ & $\Z_p\rtimes_6^d D_{2\cdot 6}$ \\
			\hline
			\hline
			$\Z_p\times C_{12}$ & 2 & 2 & 3 \\
			\hline
			$\Z_p\rtimes_6 C_{12} $ & 4 & 4 & 6 \\
			\hline
		\end{tabular}
	\end{center}
\end{proposition}

\subsubsection*{Case $E=C_6\times C_2$}

Let $E=\langle a\rangle \times \langle b \rangle$. There is 
one conjugacy class of regular subgroups of $\Hol(E)$ isomorphic to $D_{2\cdot 6}$.

\begin{center}
    \begin{tabular}{|l|l|}
    \hline
      Representative& $F=\langle r=(a,\mathrm{Id}),s=(b,g_2^3)\rangle$ \\
      \hline
      $\pi_2(F)$& $ \langle g_2^3\rangle $\\
      $\pi_2(F)\subseteq \Sigma_{\sigma}$ &  $\sigma$ of order $1$ or $2$ \\
      $\Aut(B_{12})$ & $\langle g_2^3,g_1g_2^2\rangle$\\
      Orbits in $\Hom(B_{12},\Z_p^*)$ & 1 of order 1, 2 of order 2\\
      \hline
    \end{tabular}
\end{center}    
Note that in order $2$,  the orbit of $(1,0)\in \Hom(E,\Z_p^*)$ under the action of $\Aut(E)$ gets disconnected when we consider the action of $\Aut(B_{12})$. We have $(0,1)$ as fixed point and $(1,0)=(1,1)g_1g_2^2$.

Regarding the action on $\Hom(F,\Z_p)^*$,
we have $\Aut(B_{12})\cap \Sigma_{(0,1)}=\Aut(B_{12})$ and $\tau_2=\tau_3\Phi_{g_1g_2^2}$, so that there is one orbit with dihedral kernel. In the other case of non trivial $\sigma$, we have
$\Aut(B_{12})\cap \Sigma_{(1,0)}=\langle g_2^3\rangle$ and $\tau_2, \tau_3$
are fixed points. Therefore, we have 2 orbits with dihedral kernel.

\begin{proposition}
	Let $p$ be a prime number, $p\geq 7$.
	The number of braces with additive group $N=\Z_p \rtimes C_6\times C_2$ and multiplicative group $G=\Z_p\rtimes D_{2\cdot 6}$ is as shown in the following table.
	
	\begin{center}
		\begin{tabular}{|c||c|c|c|}
			\hline
			$N \backslash G$ & $\Z_p\times D_{2\cdot 6}$ & $\Z_p\rtimes_6^c D_{2\cdot 6} $ & $\Z_p\rtimes_6^d D_{2\cdot 6}$ \\
			\hline
			\hline
			$\Z_p\times (C_6\times C_2)$ & 1 & 1 & 1\\
			\hline
			$\Z_p\rtimes_6 (C_6\times C_2) $ & 4 & 4 & 6 \\
			\hline
		\end{tabular}
	\end{center}
\end{proposition}

\subsubsection*{Case $E=D_{2\cdot6}$}

Let $E=\langle r,s\rangle$. 
There are four conjugacy classes of regular subgroups of $\mathrm{Hol}(E)$ isomorphic to $D_{2\cdot6}$. 

\begin{center}
    \begin{tabular}{|l|l|l|l|}
    
    \hline
       $F_1=\langle (r,\Id),(s,\Id)\rangle$ &
     $ F_2=\langle(rs,g_1),(s,\mathrm{Id})\rangle$&
      $F_3=\langle(r,g_2^4),(s,g_1)\rangle$&
      $F_4=\langle(s,g_2^2),(r^3s,g_1)\rangle$\\
      \hline
      $\pi_2(F_1)=\{\Id\}$&$\pi_2(F_2)=\langle g_1\rangle$ &
       $\pi_2(F_3)=\langle g_1,g_2^2\rangle$&
        $\pi_2(F_4)=\langle g_1,g_2^2\rangle$\\
      $\pi_2(F_1)\subseteq \Sigma_{\sigma}$ $\forall \sigma$ & 
      $\pi_2(F_2)\subseteq \Sigma_{\sigma}$ $\forall \sigma$ &
      $\pi_2(F_3)\subseteq \Sigma_{\sigma}$ $\forall \sigma$ &
      $\pi_2(F_4)\subseteq \Sigma_{\sigma}$ $\forall \sigma$
       \\
       $\Aut(B_{12})= \Aut(E)$&
       $\Aut(B_{12})=\langle g_1\rangle$ &$\Aut(B_{12})= \Aut(E)$&
       $\Aut(B_{12})=\langle g_1g_2^2\rangle$\\
       3 orbits & 4 orbits & 3 orbits & 4 orbits\\
      \hline
    \end{tabular}
\end{center}    

When we examine the action of $\Aut(B_{12})\cap \Sigma_{\sigma}$ on $\Hom(F_i,\Z_p^*),$ for every possible $\sigma$, 
we see that we have $\tau_2$ and $\tau_3$ in the same orbit only for $i=1,3$ and $\Sigma_{\sigma}=\Aut(E)$.

\begin{proposition}
	Let $p$ be a prime number, $p\geq 7$.
	The number of skew braces with additive group $N=\Z_p\rtimes D_{2\cdot 6}$ and multiplicative group $G=\Z_p\rtimes D_{2\cdot 6}$ is as shown in the following table.
	
	\begin{center}
		\begin{tabular}{|c||c|c|c|}
			\hline
			$N \backslash G$ & $\Z_p\times D_{2\cdot 6}$ & $\Z_p\rtimes_6^c D_{2\cdot 6} $ & $\Z_p\rtimes_6^d D_{2\cdot 6}$ \\
			\hline
			\hline
			$\Z_p\times D_{2\cdot 6}$ & 4 & 4 & 6 \\ \hline $\Z_p\rtimes_6^c D_{2\cdot 6} $ & 8 & 8 & 12 \\ \hline
			$\Z_p\rtimes_6^d D_{2\cdot 6} $ & 12 & 12 & 24 \\ \hline
		\end{tabular}
	\end{center}
\end{proposition}

\subsubsection*{Case $E=\Dic_{12}$}

Let $E=\langle x,y\rangle$. There are also  four conjugacy classes of regular subgroups of $\mathrm{Hol}(E)$ isomorphic to $D_{2\cdot6}$. 

\begin{center}
    \begin{tabular}{|l|l|l|l|}
    
    \hline
       $F_1=\langle (xy,g_1),(y^2,g_1g_2^3)\rangle$ &
     $ F_2=\langle (xy,g_2),(y^3,g_1)\rangle$&
      $F_3=\langle (xy^2,\Id),(y,g_2^3)\rangle$&
      $F_4=\langle (xy^2,g_2^4),(y,g_1)\rangle$\\
      \hline
      $\pi_2(F_1)=\langle g_1,g_2^3\rangle$ &
      $\pi_2(F_2)=\mathrm{Aut}(E)$ &
      $\pi_2(F_3)=\langle g_2^3\rangle$&
      $\pi_2(F_4)=\langle g_1,g_2^2\rangle$\\
      $\subseteq \Sigma_{\sigma}$,  $\sigma=(1,1),(2,1)$ & 
      $\subseteq \Sigma_{\sigma}$,  $\sigma=(1,1),(2,1)$ &
      $\subseteq \Sigma_{\sigma}$,  $\sigma=(1,1),(2,1)$ &
      $\subseteq \Sigma_{\sigma}$,  $\sigma=(1,1),(2,1)$
       \\
       $\Aut(B_{12})= \langle g_1g_2^3\rangle$&
       $\Aut(B_{12})=\langle g_1g_2^5\rangle$ &$\Aut(B_{12})= \Aut(E)$&
       $\Aut(B_{12})=\Aut(E)$\\
       2 orbits & 2 orbits & 2 orbits & 2 orbits\\
      \hline
    \end{tabular}
\end{center}    

Regarding the action of $\Aut(B_{12})\cap \Sigma_{\sigma}=\Aut(B_{12})$ on $\Hom(F_i,\Z_p^*)$, we have that $\tau_2,\tau_3$ are fixed points when $i=1,2$ and 
$$
\tau_2=\tau_3\Phi_{g_1g_2^2} \mbox{ for } F_3, \quad
\tau_2=\tau_3\Phi_{g_1} \mbox{ for } F_4.
$$

\begin{proposition}
	Let $p$ be a prime number, $p\geq 7$.
	The number of skew braces with additive group $N=\Z_p\rtimes Dic_{12}$ and multiplicative group $G=\Z_p\rtimes D_{2\cdot 6}$ is as shown in the following table.
	
	\begin{center}
		\begin{tabular}{|c||c|c|c|}
			\hline
			$N \backslash G$ & $\Z_p\times D_{2\cdot 6}$ & $\Z_p\rtimes_6^c D_{2\cdot 6} $ & $\Z_p\rtimes_6^d D_{2\cdot 6}$ \\
			\hline
			\hline
			$\Z_p\times \Dic_{12}$ & 4 & 4 & 6 \\ \hline $\Z_p\rtimes_6 \Dic_{12} $ & 8 & 8 & 12 \\ \hline 
		\end{tabular}
	\end{center}
\end{proposition}

\subsection{$F=\Dic_{12}$}

 According to the table in Section \ref{twelve}
and Theorem \ref{npgivesn}, for a brace of size 12 with multiplicative group  $F=\Dic_{12}$, the additive group can be any group of order 12. In all the sub-cases the group $\Hom(F,\Z_p^*)$ is as in \ref{dici} so that
$\tau$ identifies as a pair $(1,1), (2,1), (4,1)$ or $(4,3)$.

\subsubsection*{Case $E=C_{12}$}

Let $E=\langle c\rangle\simeq C_{12}$. There is  one conjugacy class of regular subgroups of $\Hol(E)$ isomorphic to $\Dic_{12}$. 

\begin{center}
    \begin{tabular}{|l|l|}
    \hline
      Representative& $F=\langle x=(c^4,\Id) ,\ y=(c^3,g_5)\rangle $ \\
      \hline
      $\pi_2(F)$ & $\langle g_5 \rangle$ \\
      $\pi_2(F)\subseteq \Sigma_{\sigma}$ &  $\sigma=(d,j)$ with $d=1,2,4$ \\
      $\Phi_g(F)=F$ & $\forall g$ since $F$ is normal in $\Hol(E)$\\
      $\Aut(B_{12})$ & $\Aut(E)$\\
      Orbits in $\Hom(B_{12},\Z_p^*)$ &  3, one for each order\\
      \hline
    \end{tabular}
\end{center}    

For $\sigma$ of order $1$ or $2$, $\Sigma_{\sigma}=\Aut(E)=\Aut(B_{12})$
and the action on $\Hom(F,\Z_p^*)$ gives one orbit for each order.
For $\sigma$ of order $4$, since $\Phi_{g_5}(y)=(c^3, g_5)=y,$
the points $(4,1)$ and $(4,3)$ are fixed.

\begin{proposition}
	Let $p$ be a prime number, $p\geq 7$. The number of braces with additive group $N=\Z_p \rtimes C_{12}$ and multiplicative group $G=\Z_p\rtimes \Dic_{12}$ is as shown in the following table, where we need
	$p\equiv 1 \pmod{\frac {12}k}$ for a kernel of size $k$ to occur.
	
	\begin{center}
		\begin{tabular}{|c||c|c|c|}
			\hline
			$N \backslash G$ & $\Z_p\times \Dic_{12}$ & $\Z_p\rtimes_{6} \Dic_{12}$ & $\Z_p\rtimes_{3} \Dic_{12}$  \\
			\hline
			\hline
			$\Z_p \times C_{12}$ &1&1&1\\
			\hline
			$\Z_p \rtimes_{6} C_{12}$ &2&2&2\\
			\hline
			$\Z_p \rtimes_{3} C_{12}$ &2&2&4\\
			\hline
			
		\end{tabular}
	\end{center}
\end{proposition}

\subsubsection*{Case $E=C_{6}\times C_2$}

Let $E=\langle a\rangle \times \langle b \rangle$. There is 
one conjugacy class of regular subgroups of $\Hol(E)$ isomorphic to $\Dic_{12}$.

\begin{center}
    \begin{tabular}{|l|l|}
    \hline
      Representative& $F=\langle x=(a^2,\Id), y=(b,g_1 g_2)\rangle $ \\
      \hline
       $\pi_2(F)$& $\langle g_2 g_1\rangle$ \\
      $\pi_2(F)\subseteq \Sigma_{\sigma}$ &  $\sigma=(1,0)$ with $d=1$, or $\sigma=(1,1)$ with $d=2$\\
      $\Phi_g(F)=F$ & $\forall g$ since $F$ is normal in $\Hol(E)$\\
      Orbits in $\Hom(B_{12},\Z_p^*)$ &  2, one for each order\\
      \hline
    \end{tabular}
\end{center}    

For the nontrivial $\sigma$ we have $\Sigma_{\sigma}=\langle g_2^3, g_2^4g_1\rangle$. From
$\Phi_{g_2 g_1}(y)=y^3$, we obtain that
$(4,3)=(4,1)\Phi_{g_2 g_1}$ and in $\Hom(F,\Z_p^*)$ there is just one orbit in order 4.

\begin{proposition}
	Let $p$ be a prime number, $p\geq 7$. The number of braces with additive group $N=\Z_p \rtimes (C_{6}\times C_2)$ and multiplicative group $G=\Z_p\rtimes \Dic_{12}$ is as shown in the following table, where we need $p\equiv 1 \pmod{4}$ for the group  $\Z_p\rtimes_{3} \Dic_{12}$ to occur.
	
	\begin{center}
		\begin{tabular}{|c||c|c|c|}
			\hline
			$N \backslash G$ & $\Z_p\times \Dic_{12}$ & $\Z_p\rtimes_{6} \Dic_{12}$ & $\Z_p\rtimes_{3} \Dic_{12}$  \\
			\hline
			\hline
			$\Z_p \times (C_{6}\times C_2)$ &1&1&1\\
			\hline
			$\Z_p \rtimes_{6} (C_{6}\times C_2)$ &2&2&2\\
			\hline
			
		\end{tabular}
	\end{center}
	
\end{proposition}

\subsubsection*{Case $E=A_4$}

There are two  conjugacy classes of regular subgroups of $\Hol(A_4)$ isomorphic to $\Dic_{12}$.

\begin{center}
    \begin{tabular}{|l|l|}
    \hline
      $F_1=\langle x_1,\ y_1\rangle $ &
     $F_2=\langle x_2, \ y_2\rangle$\\
     $x_1=((1,3,2),\phi_{(1,2,3)}),y_1=((1,3)(2,4),\phi_{(1,2)})$&
     $ x_2=((1,3,4),\Id), y_2=((1,3)(2,4),\phi_{(1,3,2,4)})$\\\hline
      $\pi_2(F_1)=\langle \phi_{(1,2,3)},\phi_{(1,2)}\rangle$ & $\pi_2(F_2)=\langle \phi_{(1,3,2,4)} \rangle $  \\
      $\pi_2(F_1)\subseteq \Sigma_{\sigma}$ for trivial  $\sigma$&
      $\pi_2(F_2)\subseteq \Sigma_{\sigma}$ for trivial  $\sigma$\\
      \hline
    \end{tabular}
\end{center}

Since
$$\Phi_{\phi_{(1,2)}}(x_1)=x_1^2,\quad \Phi_{\phi_{(1,2)}}(y_1)=y_1^3,\quad \Phi_{\phi_{(1,2)}}(F_1)=F_1,\quad (4,3)=(4,1)\Phi_{\phi_{(1,2)}} $$
$$\Phi_{\phi_{(3,4)}}(x_2)=x_2^2,\quad \Phi_{\phi_{(3,4)}}(y_2)=y_2^3,
\quad\Phi_{\phi_{(3,4)}}(F_2)=F_2, \quad (4,3)=(4,1)\Phi_{\phi_{(3,4)}},$$
in $\Hom(F_i,\Z_p^*)$ there is one orbit for each order.

\begin{proposition}
	Let $p$ be a prime number, $p\geq 7$. The number of skew braces with additive group $N=\Z_p \rtimes A_4$ and multiplicative group $G=\Z_p\rtimes \Dic_{12}$ is as shown in the following table, where we need $p\equiv 1 \pmod{4}$ for the group  $\Z_p\rtimes_{3} \Dic_{12}$ to occur.

	\begin{center}
		\begin{tabular}{|c||c|c|c|}
			\hline
			$N \backslash G$ & $\Z_p\times \Dic_{12}$ & $\Z_p\rtimes_{6} \Dic_{12}$ & $\Z_p\rtimes_{3} \Dic_{12}$ \\
			\hline
			\hline
			$\Z_p \times A_4$ &2&2&2\\
			\hline
			
		\end{tabular}
	\end{center}
\end{proposition}

\subsubsection*{Case $E=D_{2\cdot 6}$}

Let $E=\langle r, s\rangle$. There are two conjugacy classes of regular subgroups of $\Hol(E)$ isomorphic to $\Dic_{12}$.

\begin{center}
    \begin{tabular}{|l|l|l|}
    \hline
     Representative & $F_1=\langle x_1=(r^2,\Id),\ y_1=(s r^5,g_2^3)\rangle $ &
     $F_2=\langle x_2=(r^2,g_2^2),\ y_2=(r^2 s,g_2 g_1)\rangle$\\
   \hline
   $\pi_2(F)$
     & $\langle g_2^3\rangle$ & 
     $\langle g_2^2,g_2g_1\rangle$ \\
     $\pi_2(F)\subseteq \Sigma_{\sigma}$& 
     $\sigma$ of order 1 or 2 &  $\sigma$ of order 1 or 2\\
     Orbits in $\Hom(B_{12},\Z_p^*)$ & 2, one for each order &2, one for each order \\ \hline
    \end{tabular}
\end{center}    

We have
$$
\Phi_{g_2^3}(x_i)=x_i,\quad \Phi_{g_2^3}(y_i)=y_i^3,
\quad \Phi_{g_2^3}(F_i)=F_i, \quad (4,3)=(4,1)\Phi_{g_2^3},
\quad i=1,2.$$

So, for each $F_1, F_2$ we obtain one orbit for each order.

\begin{proposition}
	Let $p$ be a prime number, $p\geq 7$. The number of skew braces with additive group $N=\Z_p \rtimes D_{2\cdot6}$ and multiplicative group $G=\Z_p\rtimes \Dic_{12}$ is as shown in the following table, where we need $p\equiv 1 \pmod{4}$ for the group  $\Z_p\rtimes_{3} \Dic_{12}$ to occur.
	
	\begin{center}
		\begin{tabular}{|c||c|c|c|}
			\hline
			$N \backslash G$ & $\Z_p\times \Dic_{12}$ & $\Z_p\rtimes_{6} \Dic_{12}$ & $\Z_p\rtimes_{3} \Dic_{12}$  \\
			\hline
			\hline
			$\Z_p \times D_{2\cdot6}$ &2&2&2\\
			\hline
			$\Z_p \rtimes_6^c D_{2\cdot6}$ &4&4&4\\
			\hline

		\end{tabular}
	\end{center}
	
\end{proposition}

\subsubsection*{Case $E=\Dic_{12}$}

Let $E=\langle x, y\rangle$. There are two conjugacy classes of regular subgroups of $\Hol(E)$ isomorphic to $\Dic_{12}$.

\begin{center}
    \begin{tabular}{|l|l|l|}
    \hline
     Representative & $F_1=\langle x_1=(x,\Id),\  y_1=(y,\Id)\rangle$ &
     $F_2=\langle x_2=(x^2,g_2^2),\ y_2=(x^2 y,g_2 g_1)\rangle$\\
   \hline
   $\pi_2(F)$
     & $\{\Id \}$ & 
     $\langle g_2^2, g_2 g_1\rangle=\langle g_2^2, g_1g_2\rangle$ \\
     $\pi_2(F)\subseteq \Sigma_{\sigma}$& 
     $\forall \sigma$  &  $\forall\sigma$ \\
     $\Aut(B_{12})$ & contains $g_1$ & contains $g_1$\\
     Orbits in $\Hom(B_{12},\Z_p^*)$ & 3, one for each order &3, one for each order \\ \hline
    \end{tabular}
\end{center}    

We have
$$
\Phi_{g_1}(x_i)=x_i^2,\quad \Phi_{g_1}(y_i)=y_i^3,
\quad \Phi_{g_1}(F_i)=F_i, \quad i=1,2,
$$
and $(4,3)=(4,1)g_1 \mbox{ in } \Hom(F,\Z_p^*).$ For $\sigma=(2,1)$, since $\Sigma_{\sigma}=\Aut(E)$, the above computation shows
that  $(4,3)=(4,1)\Phi_{g_1}$ and we have just one orbit of order 4.
But for $\sigma$ of order 4, since $\Sigma_{\sigma}$ is also the stabilizer of the $\tau$'s of order 4, we have fixed points and 2 orbits.

\begin{proposition}
	Let $p$ be a prime number, $p\geq 7$.
	The number of skew braces with additive group $N:=\Z_p \rtimes \Dic_{12}$ and multiplicative group $G:=\Z_p\rtimes \Dic_{12}$ is as shown in the following table, where we need $p\equiv 1 \pmod{\frac {12}k}$ for a kernel of size $k$ to occur.
	
	\begin{center}
		\begin{tabular}{|c||c|c|c|}
			\hline
			$N \backslash G$ & $\Z_p\times \Dic_{12}$ & $\Z_p\rtimes_{6} \Dic_{12}$ & $\Z_p\rtimes_{3} \Dic_{12}$  \\
			\hline
			\hline
			$\Z_p\times \Dic_{12}$ &2&2&2\\
			\hline
			$\Z_p\rtimes_{6} \Dic_{12}$ &4&4&4\\
			\hline
			$\Z_p\rtimes_{3} \Dic_{12}$ &4&4&8\\
			\hline
			
		\end{tabular}
	\end{center}
\end{proposition}

\subsection{Total numbers}

For a prime number $p\geq 7$ we compile in the following tables the total number of skew braces of size $12p$. The additive group is a semidirect product $\Z_p\rtimes E$ and the multiplicative group is a semidirect product $\Z_p\rtimes F$.
In the first column we have the possible $E$'s and in the first row the possible $F$'s.

\begin{itemize}
	\item If $p\equiv 11 \pmod{12}$
	\begin{center}
		\begin{tabular}{r|c|c|c|c|c|c}
			&\, \, $C_{12}$ \, \,   &   $C_6\times C_2$   & \, \,  $A_4$ \, \,   & \, \, $D_{2\cdot 6}$\, \,  & \, $\Dic_{12}$ \, \, & \\
			\hline
			$C_{12}$        & 6 &9 & 0 & 21 & 6&\\
			$C_6\times C_2 $ & 6 & 8& 1& 17 &6&\\
			$A_4$           & 0 &4 & 4 & 0 &4 &\\
			$D_{2\cdot 6}$   & 12 &34 &0 & 90 &12&\\
			$\Dic_{12}$     & 12 &18 &0 & 42 &12&\\
			
			\hline
			&36&73  & 5 & 170 & 40 & $\mathbf{324}$\\
		\end{tabular}
	\end{center}
	
\item If $p\equiv 5 \pmod{12}$
	\begin{center}
		\begin{tabular}{r|c|c|c|c|c|c}
			&\, \, $C_{12}$ \, \,   &   $C_6\times C_2$   & \, \,  $A_4$ \, \,   & \, \, $D_{2\cdot 6}$\, \,  & \, $\Dic_{12}$ \, \, & \\
			\hline
			$C_{12}$        &17 &9 &0  & 21 &17 &\\
			$C_6\times C_2$  & 9 &8 &1 & 17 &9&\\
			$A_4$           &0 & 4& 4 & 0 &6 &\\
			$D_{2\cdot 6}$   & 18 &34 &0 & 90 & 18&\\
			$\Dic_{12}$     & 34 &18 &0 & 42 &34&\\
			
			\hline
			&78& 73 & 5 & 170 & 84 & $\mathbf{410}$\\
		\end{tabular}
	\end{center}
	
	\item If $p\equiv 7 \pmod{12}$
	\begin{center}
		\begin{tabular}{r|c|c|c|c|c|c}
			&\, \, $C_{12}$ \, \,   &   $C_6\times C_2$   & \, \,  $A_4$ \, \,   & \, \, $D_{2\cdot 6}$\, \,  & \, $\Dic_{12}$ \, \, & \\
			\hline
			$C_{12}$   &36 &54 & 0 & 21 & 6&\\
			$C_6\times C_2$  &36 & 46& 8& 17 &6&\\
			$A_4$   &0 & 32& 32 & 0 &4 &\\
			$D_{2\cdot 6}$  &24 &68 &0 & 90 &12&\\
			$\Dic_{12}$  & 24&36 & 0& 42 &12&\\
			
			\hline
			&120&236  & 40 & 170 &40  & $\mathbf{606}$\\
		\end{tabular}
	\end{center}
	
	\item If $p\equiv 1 \pmod{12}$
	\begin{center}
		\begin{tabular}{r|c|c|c|c|c|c}
			&\, \, $C_{12}$ \, \,   &   $C_6\times C_2$   & \, \,  $A_4$ \, \,   & \, \, $D_{2\cdot 6}$\, \,  & \, $\Dic_{12}$ \, \, & \\
			\hline
			$C_{12}$   &94 & 54& 0 & 21 &17 &\\
			$C_6\times C_2$  &54 &46 &8 & 17 &9&\\
			$A_4$   &0 &32 & 32 & 0 &6 &\\
			$D_{2\cdot 6}$  & 36 & 68&0 & 90 &18&\\
			$\Dic_{12}$    &68  &36 &0 & 42 &34&\\
			
			\hline
			&252& 236 & 40 & 170 & 84 & $\mathbf{782}$\\
		\end{tabular}
	\end{center}
	
\end{itemize}

With the results summarized in the above tables, the validity of conjecture  \eqref{conjec} is then established.

\end{document}